\documentclass[reqno]{amsart}

\usepackage{amsmath}
\usepackage{amssymb}
\usepackage[margin=1in]{geometry}
\usepackage{xcolor}
\usepackage{comment}
\DeclareMathOperator{\loc}{loc}

\DeclareMathOperator{\osc}{osc}

\DeclareMathOperator{\tr}{tr}
\newcommand{\norm}[1]{\left\lVert#1\right\rVert}

\theoremstyle{plain}
\newtheorem{thm}{Theorem}[section]
\newtheorem{lem}{Lemma}[section]

\newtheorem{cor}{Corollary}[section]

\theoremstyle{definition}
\newtheorem{defn}{Definition}[section]

\theoremstyle{remark}

\usepackage{cite}
\usepackage[colorlinks, pdfpagelabels, pdfstartview = FitH,bookmarksopen = true, bookmarksnumbered = true,linkcolor = blue,plainpages = false,hypertexnames = false, citecolor = red]{hyperref}
\usepackage[italian, textsize=tiny, color=BlueGreen, backgroundcolor=BlueGreen, linecolor=BlueGreen]{todonotes}
\usepackage{enumitem}
\usepackage{nameref}
\makeatletter
\let\orgdescriptionlabel\descriptionlabel
\renewcommand*{\descriptionlabel}[1]{%
  \let\orglabel\label
  \let\label\@gobble
  \phantomsection
  \edef\@currentlabel{#1\unskip}%
  \let\label\orglabel
  \orgdescriptionlabel{#1}%
}
\numberwithin{equation}{section}

\begin{document}
	
\title[$W^{2,\delta}$ type estimates for degenerate fully nonlinear elliptic equations]{Interior $W^{2,\delta}$ type estimates for degenerate fully nonlinear elliptic equations with $L^n$ data }
%%%%% AUTHORS %%%%%

\author[Byun]{Sun-Sig Byun}
\address{Department of Mathematical Sciences and Research Institute of Mathematics,
	Seoul National University, Seoul 08826, Republic of Korea}
\email{byun@snu.ac.kr}

\author[Kim]{Hongsoo Kim}
\address{Department of Mathematical Sciences, Seoul National University, Seoul 08826, Republic of Korea}
\email{rlaghdtn98@snu.ac.kr}

\author[Oh]{Jehan Oh}
\address{Department of Mathematics, Kyungpook National University,
Daegu 41566, Republic of Korea}
\email{jehan.oh@knu.ac.kr}

\thanks {S.-S. Byun was supported by the National Research Foundation of Korea(NRF) grant funded by the Korea government [Grant No. 2022R1A2C1009312]. H. Kim was supported by the National Research Foundation of Korea(NRF) grant funded by the Korea government [Grant No. 2021R1A4A1027378]. J. Oh was supported by the National Research Foundation of Korea(NRF) grant funded by the Korea government [Grant Nos. 2020R1C1C1A01014904 and RS-2023-00217116].}

\makeatletter
\@namedef{subjclassname@2020}{\textup{2020} Mathematics Subject Classification}
\makeatother
% \date{\today}
% It is required to enter 2010 MSC.
% Please provide minimum  5 keywords.

\everymath{\displaystyle}

\begin{abstract}
	We establish interior $W^{2,\delta}$ type estimates for a class of degenerate fully nonlinear elliptic equations with $L^n$ data. The main idea of our approach is to slide $C^{1,\alpha}$ cones, instead of paraboloids, vertically to touch the solution, and estimate the contact set in terms of the measure of the vertex set. This shows that the solution has tangent $C^{1,\alpha}$ cones almost everywhere, which leads to the desired Hessian estimates. Accordingly, we are able to develop a kind of counterpart to the estimates for divergent structure quasilinear elliptic problems, as discussed in \cite{Cianchi18,Mingione18}.
\end{abstract}

\maketitle

\section{Introduction}
\label{sec1}
In this paper, we study interior $W^{2,\delta}$ type estimates for a viscosity solution $u$ of degenerate fully nonlinear elliptic inequalities
\begin{align} \label{PDE}
|Du|^{\gamma} \mathcal{P}_{\lambda,\Lambda}^{-}(D^{2}u) \leq f(x) \leq |Du|^{\gamma} \mathcal{P}_{\lambda,\Lambda}^{+}(D^{2}u)  \quad\text{in}\ B_{1},
\end{align}
where $\gamma \geq 0 $, $B_1 \subset \mathbb{R}^n$ with $n\geq 2$ is the unit ball, $\mathcal{P}_{\lambda,\Lambda}^{\pm}$ are the Pucci extremal operators with ellipticity constants $0 < \lambda \leq \Lambda < \infty$, and $f \in C \cap L^n(B_1)$.

The inequalities \eqref{PDE} include the following important types of equations:
\begin{enumerate}
    \item The degenerate fully nonlinear elliptic equation
\begin{align}
\label{DEQ}
    |Du|^\gamma F(D^2u) = f \quad\text{in}\ B_{1},
\end{align}
where $\gamma \geq 0$, $F$ is uniformly elliptic, and $f \in C \cap L^n(B_1)$.
    \item The degenerate $p$-Laplacian equation
\begin{align}
\label{pLap}
    \Delta_p u = \mathrm{div}(|Du|^{p-2}Du) = f \quad\text{in}\ B_{1},
\end{align}
where $p = \gamma +2 \geq 2$ and $f \in C \cap L^n(B_1)$.
\end{enumerate}

There has been lots of progress in the regularity of these types of equations in recent years.
For \eqref{DEQ}, Imbert and Silvestre \cite{Imbert13} established interior $C^{1,\alpha}$ estimates with $f \in L^\infty$ and Araújo, Ricarte, and Teixeira \cite{Araujo15} proved optimal $C^{1,\alpha}$ estimates with $\alpha = \frac{1}{1+\gamma}$ when $F$ is concave.
For more results regarding \eqref{DEQ}, we refer the reader to \cite{Andrade22,Yun24,Li232,Pimentel24,Silva21,Araujo23,Banerjee22,Davila09,Imbert11,Urbano23}, and for more results related to nonstandard degeneracy like variable exponent degeneracy or double phase degeneracy, we refer to \cite{Baasandorj22,Baasandorj24,Baasandorj23,Teixeira20,Jesus22,Fang21,Filippis22,Filippis21} and the references therein.

For the uniformly elliptic case that $\gamma =0$, Lin \cite{Lin86} first proved interior $W^{2,\delta}$ estimates for linear elliptic nondivergent equations, and Caffarelli and Cabr\'e \cite{Caffarelli95} obtained interior $W^{2,\delta}$ estimates for viscosity solutions.
These $W^{2,\delta}$ estimates turned out to be helpful for deriving interior $W^{2,p}$ estimates when $p>n$, $f \in L^p$ and the homogeneous equation $F(D^2h) =0$ has $C^{1,1}$ estimates.
Moreover, they are also used to prove partial regularity in \cite{Armstrong12,Daniel15} and convergence of blow down solutions, see \cite{Yuan01}.
For the singular case that $-1 < \gamma \leq 0$, D. Li and Z. Li \cite{Li18, Li17} proved $W^{2,\delta}$ estimates by using a sliding paraboloid method.
For the treatment of a more general case, we refer the reader to \cite{Baasandorj242,li23}.
However, for the degenerate case $\gamma >0$, it is not easy to prove $W^{2,\delta}$ estimates because the equation degenerates and the Hessian $D^2u$ can blow up when the gradient $Du$ is very small.
To the best of our knowledge, there are no results about the integrability of the Hessian of solutions of degenerate fully nonlinear elliptic equations.

On the other hand, for the degenerate $p$-Laplacian \eqref{pLap} which is in divergence form, Lou \cite{Lou08} proved that a weak solution $u$ satisfies
    $$|Du|^{p-2}Du \in W^{1,2}_{\loc}$$
when $f \in L^s$ with $s > \max\{2,n/p\}$.
For a further discussion of these kinds of regularity, we refer the reader to \cite{Canino18,Cianchi18,Mingione18}.
Since the approach in \cite{Lou08} can be used only for the divergence form as far as we are concerned, we cannot directly apply the same method to viscosity solutions in nondivergence form. Here we use a sliding $C^{1,\alpha}$ cone method to parallel the nondivergent counterpart of the result in \cite{Lou08}. More precisely, we prove that any viscosity solution $u$ of \eqref{PDE} satisfies
    $$|Du|^{\gamma}Du \in W^{1,\delta}_{\loc}$$
for some universal constant $\delta>0$ for a given $f \in L^n$.
Note that if $\gamma=0$, then this is the same as $W^{2,\delta}$ estimate for the uniformly elliptic case as in \cite{Lin86} and \cite{Caffarelli95}.

The main result of this paper is the following:
\begin{thm} \label{main}
Let $u \in C(\overline{B}_{1})$ satisfy
\begin{align*}
	|Du|^{\gamma} \mathcal{P}_{\lambda,\Lambda}^{-}(D^{2}u) \leq f(x) \leq |Du|^{\gamma} \mathcal{P}_{\lambda,\Lambda}^{+}(D^{2}u) \quad\text{in}\ B_{1}
\end{align*}
in the viscosity sense, with $0 < \lambda \leq \Lambda < \infty$, $\gamma \geq 0$ and $f \in C \cap L^{n}(B_{1})$. Then $|Du|^\gamma Du \in W^{1,\delta}(B_{1/2})$ with the estimate
    $$\norm{|Du|^{\gamma}Du}_{W^{1, \delta}(B_{1/2})} \leq C \left( \norm{u}_{L^{\infty}(B_{1})}^{1+\gamma} + \norm{f}_{L^{n}(B_{1})} \right),$$
where $\delta>0$ and $C>0$ depend only on $n, \lambda, \Lambda$ and $\gamma$.
\end{thm}

Our proof is influenced by the proof of  $W^{2,\delta}$ estimates for the singular case by D. Li and Z. Li \cite{Li18}, which uses localization, a covering argument, and a sliding paraboloid method.
Roughly speaking, the main idea of sliding paraboloid method is sliding paraboloids from below and above in order to touch the solution, thereby finding the lower bound of the measure of the set of touching points in terms of the measure of the set of vertex points using the area formula.
This method was first introduced by Cabr\'e \cite{Cabre97} and developed by Savin \cite{Savin07}.
For more results regarding this method, we refer the reader to \cite{Imbert16,Mooney15,Pimentel22} and \cite{Colombo14}.
However, it seems that a paraboloid does not fit well with degenerate inequalities like \eqref{PDE} because of the following reasons.
First, the scaling of a paraboloid does not match well with the scaling of degenerate inequalities. 
Second, the Hessian of a paraboloid is a constant no matter what the gradient is, while the Hessian of a solution can blow up when the gradient is small and so the equation degenerates.
Finally, there is a difficulty in using a paraboloid along with the area formula to estimate the measure of the set of touching points in terms of the measure of the set of vertex points due to the degeneracy of the inequalities.

The main originality of this paper is to slide a $C^{1,\alpha}$ cone $P(x)=\frac{1}{1+\alpha}|x|^{1+\alpha}$ where $\alpha = \frac{1}{1+\gamma}$, instead of a paraboloid $\frac{1}{2}|x|^2$, which enables us to solve the difficulty we mentioned above.  
We remark first that the choice of $\alpha$ matches with the scaling of a solution.
In the degenerate case, it is natural to consider the scaling $\tilde{u}(x) = \frac{u(rx)}{r^{1+\alpha}}$ with $0<r<1$, since $\tilde{u}$ satisfies the same inequalities \eqref{PDE} with $f(x)$ replaced by $\tilde{f}(x) = f(rx)$.
Notice that the same scaling of the $C^{1,\alpha}$ cone $\tilde{P}(x) = \frac{P(rx)}{r^{1+\alpha}}$ is the same as the original $C^{1,\alpha}$ cone $P(x)$.
Next, observe that $P$ satisfies $|DP|^\gamma \Delta P = c$ for some constant $c$, therefore the Hessian of $P$ blows up when the gradient of $P$ is small.
Finally, $C^{1,\alpha}$ cones fit exactly with the degeneracy of the equation when using the area formula.
If we slide a concave $C^{1,\alpha}$ cone $P(x) = -\frac{1}{1+\alpha}|x-y_0|^{1+\alpha}$ with a vertex $y_0$, so that $P$ touches a $C^2$ function $u$ from below at a touching point $x_0$, then we have the equation
$$y_0 = x_0 + |Du|^\gamma Du(x_0).$$
Differentiating the equation with respect to $x_0$ and taking a trace, we find
$$\tr D_{x_0}y_0 = n + \mathrm{div}(|Du|^\gamma Du)(x_0) = n + \Delta_{\gamma+2} u (x_0).$$
Heuristically, if we assume that $u$ is a classical solution of \eqref{pLap}, we say that $\Delta_{\gamma+2} u (x_0) \leq f(x_0)$.
By using a bound of $\tr D_{x_0}y_0$, we can use the area formula to obtain a lower bound of the measure of the set of touching points.
Note that since $P$ mentioned above satisfies $D(|DP|^\gamma DP) = cI$ where $I$ is the identity matrix, it is possible to bound $D(|Du|^\gamma Du)$ at the point where $C^{1,\alpha}$ cones touches $u$ from below and above.

However, unlike paraboloids, $C^{1,\alpha}$ cones cause some difficulties.
First, the growth and degeneracy of a $C^{1,\alpha}$ cone is different from that of a paraboloid. The Hessian of a $C^{1,\alpha}$ cone blows up at the vertex point and diminishes at a point far from the vertex point, which makes us delicately consider the distance between the vertex point and the touching point.
Therefore, in the proof of the key lemma, Lemma \ref{Key}, we will divide it into three cases by the distance.
Second, the sum and difference of two $C^{1,\alpha}$ cones are not a $C^{1,\alpha}$ cone, and the sum of a $C^{1,\alpha}$ cone and a linear function is not a $C^{1,\alpha}$ cone, either.
This disadvantage makes it difficult to analyze the difference of $C^{1,\alpha}$ cones in Step 2 of the key lemma. Thus, we will instead consider the Hessian of its difference and approximate it by some function which is relatively easier to analyze.

For any point $x_0 \in B_{1/2}$, we consider the pointwise $C^{1,\alpha}$ seminorm at $x_0$ defined by
\begin{align*}
    [u]_{C^{1,\alpha}}(x_{0}) := \sup_{r>0} \inf_{l \in \mathcal{P}_{1}} \frac{\norm{u-l}_{L^{\infty}(B_{r}(x_{0}))}}{r^{1+\alpha}}.
\end{align*}
Note that if $\norm{[u]_{C^{1,\alpha}}(\cdot)}_{L^{\infty}(B_{1/2})} \leq c$, then $u \in C^{1,\alpha}(B_{1/2})$ with $\norm{u}_{C^{1,\alpha}} \leq c$.
If there exist $C^{1,\alpha}$ cones of the opening $K$ touching $u$ at $x_0$ from below and above, then  $[u]_{C^{1,\alpha}}(x_{0}) \leq cK$ for some constant $c=c(n,\alpha)$.
Thus, if we prove the power decay of the measure of the set of points which do not have touching $C^{1,\alpha}$ cones, then we have the following corollary.

\begin{cor}
    Under the assumption of Theorem \ref{main}, the point-wise $C^{1,\alpha}$ norm of $u$ is integrable, in the sense that $[u]_{C^{1,\alpha}}(\cdot) \in L^{\delta}(B_1)$ where $\alpha = \frac{1}{1+\gamma}$.\\ In particular, $[u]_{C^{1,\alpha}}(\cdot)$ is almost everywhere finite, which implies that $u$ is almost everywhere $C^{1,\alpha}$.
\end{cor}
The authors of \cite{Araujo15} proved that if $u$ is a solution of \eqref{DEQ} and $F$ is concave, then $[u]_{C^{1,\alpha}}(\cdot) \in L^{\infty}$ whose norm depends on $L^\infty$ data of $f$.
Moreover, the authors of \cite{Urbano17} obtained that if $u$ is a solution of \eqref{pLap} in the plane, then $[u]_{C^{1,\alpha}}(\cdot) \in L^{\infty}$ whose norm depends on $L^\infty$ data of $f$.
However, this corollary implies that if $u$ is a solution of \eqref{PDE}, we have  $[u]_{C^{1,\alpha}}(\cdot) \in L^{\delta}$ whose norm depends on $L^n$ data of $f$.
Especially, even for the singular set $\{Du=0\}$ where the equation degenerates, we know that almost every singular point has a tangent $C^{1,\alpha}$ cone and that the point-wise $C^{1,\alpha}$ norm is finite.

The paper is organized as follows.
In Section \ref{sec2} we specify the notations, introduce some preliminary lemmas with a covering lemma, and reduce Theorem \ref{main} to Lemma \ref{Smallness} by normalization.
In Section \ref{sec3} we consider the relation of $|Du|^\gamma Du$ with a touching $C^{1,\alpha}$ cone and reduce Lemma \ref{Smallness} to Theorem \ref{tangent}.
In Section \ref{sec4} we prove the key density lemma and the measure decay lemma to prove Theorem \ref{tangent}. 

\section{Preliminaries}
\label{sec2}
Throughout the paper, we denote by $B_r(x_0) = \{ x \in \mathbb{R}^n : |x-x_0| <r \}$ and $B_r = B_r(0)$. $S(n)$ denotes the space of symmetric $n \times n$ real matrices and $I$ denotes the identity matrix.

For given two functions $u, v : \Omega \rightarrow \mathbb{R}$ and a point $x_0 \in \Omega$, we denote by $u \overset{x_0}{\leq} v$ in $\Omega$ if $u(x) \leq v(x) $ in $\Omega$ and $u(x_0) = v(x_0)$. 

For $0<\alpha\leq 1$, we say that $P$ is a concave $C^{1,\alpha}$ cone of opening $K>0$ and vertex $y_{0}$ if
$$P(x)=-\frac{K}{1+\alpha}|x-y_{0}|^{1+\alpha} + C$$
for some constant $C$.

Now we give the definition of the set of touching points. For a given continuous function $u : \Omega \rightarrow \mathbb{R} $ and a closed set $V \subset \mathbb{R}^{n}$, we slide a concave $C^{1,\alpha}$ cone of opening $K>0$ and vertex $y_{0} \in V$ from below the graph of $u$ until it touches the graph of $u$ first time.
Then we define $T_{K}^{-}(u,V)$ to be the set of touching points.
We also define $T_{K}^{+}(u,V)$ to be the set of points where $u$ touches with a convex $C^{1,\alpha}$ cone with vertex included in $V$ from above.
\begin{defn}
For a given continuous function $u : \Omega \rightarrow \mathbb{R} $, a closed set $V \subset \mathbb{R}^{n}$ and $K>0$,
\begin{align*}
    T_{K}^{-}(u,V) :=& \left\{x_{0} \in  \Omega : \exists P\text{ concave  $C^{1,\alpha}$ cone of opening $K$ whose vertex is in $V$ such that $P(x) \overset{x_0}{\leq} u(x)$ in $\Omega$}\right\}\\
    =& \left\{x_{0} \in  \Omega : \exists y_{0} \in V \text{ such that } -\frac{K}{1+\alpha}|x-y_{0}|^{1+\alpha} + \left(u(x_0) + \frac{K}{1+\alpha}|x_0-y_{0}|^{1+\alpha}\right) \overset{x_0}{\leq} u(x) \text{ in } \Omega \right\},
\end{align*}  
\begin{gather*}
    T_{K}^{+}(u,V) :=T_{K}^{-}(-u,V),\\
    T_{K}(u,V) :=T_{K}^{-}(u,V) \cap T_{K}^{+}(u,V).
\end{gather*} 
\end{defn}
We shorten $T_{K}^{\pm} = T_{K}^{\pm}(u,\overline{\Omega})$ when they are clear in the context.
Observe that $T_{K}^{\pm}$ are closed in $\Omega$ and that if $x_{0} \in T_{K}$, then $[u]_{C^{1,\alpha}}(x_{0}) \leq cK$ for some $c=c(\alpha)$.    

For given $0<\lambda \leq \Lambda$, we recall the definition of the Pucci extremal operators $\mathcal{P}_{\lambda,\Lambda}^{\pm} : S(n) \rightarrow \mathbb{R}$ as follows (see \cite{Caffarelli95}):
\begin{align*}
    \mathcal{P}_{\lambda,\Lambda}^{+}(M) := \lambda \sum_{e_i(M)<0}e_i(M) + \Lambda \sum_{e_i(M)>0}e_i(M), \\
    \mathcal{P}_{\lambda,\Lambda}^{-}(M) := \Lambda \sum_{e_i(M)<0}e_i(M) + \lambda \sum_{e_i(M)>0}e_i(M),
\end{align*}
where $M \in S(n)$ and $e_i(M)$'s are the eigenvalues of $M$.

We remark the following properties of the Pucci extremal operators:
\begin{enumerate}
    \item For every $M,N \in S(n)$, $\mathcal{P}^{-}(M) + \mathcal{P}^{-}(N) \leq \mathcal{P}^{-}(M+N) \leq \mathcal{P}^{-}(M) + \mathcal{P}^{+}(N)$,
    \item If $M\geq 0$, then $\mathcal{P}^{-}(M) = \lambda \tr M$ and $\mathcal{P}^{+}(M) = \Lambda \tr M$.
\end{enumerate}

Now we recall the definition of the inequalities \eqref{PDE} in the viscosity sense from \cite{Ishii92, Caffarelli95} as follows.
\begin{defn}
Let $f \in C(B_1)$.
We say that $u \in C(\overline{B}_1)$ satisfies 
$$|Du|^{\gamma} \mathcal{P}_{\lambda,\Lambda}^{-}(D^{2}u) \leq f(x) \quad\text{in}\ B_{1}$$
in the viscosity sense, if for any $x_0 \in B_1$ and test function $\psi \in C^2(B_1)$ such that $u-\psi$ has a local minimum at $x_0$, then
$$|D\psi(x_0)|^{\gamma} \mathcal{P}_{\lambda,\Lambda}^{-}(D^{2}\psi(x_0)) \leq f(x_0).$$
Similarly, we say that $u \in C(\overline{B}_1)$ satisfies 
$$|Du|^{\gamma} \mathcal{P}_{\lambda,\Lambda}^{+}(D^{2}u) \geq f(x) \quad\text{in}\ B_{1}$$
in the viscosity sense, if for any $x_0 \in B_1$ and test function $\psi \in C^2(B_1)$ such that $u-\psi$ has a local maximum at $x_0$, then
$$|D\psi(x_0)|^{\gamma} \mathcal{P}_{\lambda,\Lambda}^{+}(D^{2}\psi(x_0)) \geq f(x_0).$$    
\end{defn}

For $g \in L^1(\Omega)$, we define the Hardy-Littlewood maximal function of $g$ as follows:
\begin{align*}
    \mathcal{M}(g)(x) := \sup_{r>0} \frac{1}{|B_r|} \int_{B_r(x)\cap \Omega} |g(y)| \, dy.
\end{align*}
Then the maximal function is scaling invariant, which means 
\begin{align*}
    \mathcal{M}(g_r)(x) = \mathcal{M}(g)(rx),
\end{align*}
where $g_r(x) = g(rx)$ is the scaled function of $g$.
Moreover, the maximal operator $\mathcal{M}$ has the weak type $(1,1)$ property:
\begin{align*}
    |\{x \in \Omega : \mathcal{M}(g)(x) > t \}| \leq Ct^{-1} \norm{g}_{L^1(\Omega)}, \ \, \forall t>0. 
\end{align*}
We will also use an equivalent description of $L^p$.
\begin{lem}
    \cite{Caffarelli95} Let $g$ be a nonnegative and measurable function in a bounded domain $\Omega$. Let $\eta>0$, $M>1$, and $0<p<\infty$. Then,
    $$g \in L^p(\Omega) \ \Longleftrightarrow \ S=\sum_{k=1}^{\infty} M^{pk}|\{x \in \Omega : g(x) > \eta M^k\}| < \infty$$
    and
    $$C^{-1}S \leq \norm{g}^p_{L^p(\Omega)} \leq C(S+|\Omega|),$$
    where $C>1$ is a constant depending only on $\eta, M$ and $p$.
\end{lem}
We introduce the following Vitali-type covering lemma.
\begin{lem} \label{covering}
    (\cite{Li18}, ($\theta, \Theta$)-type Covering Lemma). Let $E \subset F \subset B_1$ be measurable sets and $0<\theta<\Theta<1$ such that\\
    (i) $|E|>\theta|B_1|$,\\
    (ii) for any ball $B \subset B_1$, if $|B \cap E| \geq \theta|B|$, then $|B \cap F| \geq \Theta|B|$.\\
    Then
    \begin{align*}
     |B_1 \setminus F| \leq \left( 1-\frac{\Theta-\theta}{5^n} \right) |B_1 \setminus E|.   
    \end{align*}
\end{lem}

Finally, note that it suffices to assume that $\norm{u}_{L^{\infty}} \leq 1/16$, $\norm{f}_{L^{n}} \leq \epsilon_{0}$ for small enough $\epsilon_{0}$ to prove Theorem \ref{main} by normalization. In particular, we only need to prove the following lemma.

\begin{lem} \label{Smallness}
Let $\gamma \geq 0$. Assume $u \in C(\overline{B}_{2})$ satisfies 
\begin{align*}
|Du|^{\gamma} \mathcal{P}_{\lambda,\Lambda}^{-}(D^{2}u) \leq f \leq |Du|^{\gamma} \mathcal{P}_{\lambda,\Lambda}^{+}(D^{2}u) \quad\text{in}\ B_{2}
\end{align*}
in the viscosity sense. Then there exist constants $\delta=\delta(n,\lambda, \Lambda,\gamma)>0$ and $\epsilon_{1} = \epsilon_{1}(n,\lambda, \Lambda,\gamma)>0$ such that if $\norm{f}_{L^{n}(B_{2})} \leq \epsilon_{1}$ and $\norm{u}_{L^{\infty}(B_2)} \leq 1/16$, then 
$$\norm{|Du|^{\gamma}Du}_{W^{1, \delta}(B_{1})} \leq C,$$
where $C=C(n,\lambda,\Lambda,\gamma)>0$.
\end{lem}

We prove Theorem \ref{main} assuming Lemma \ref{Smallness}. Suppose that $u$ satisfies the assumption of Theorem \ref{main}, Let
\begin{align*}
    a = 16 \norm{u}_{L^{\infty}(B_2)} + \left( \epsilon_1 \norm{f}_{L^n(B_2)} \right)^{\frac{1}{1+\gamma}} +\epsilon
\end{align*}
for $\epsilon>0$. Then $\tilde{u}(x) := \frac{1}{a}u(x)$ and $\tilde{f}(x) := \frac{1}{a^{1+\gamma}}f(x)$ satisfy
\begin{align*}
    |D\tilde{u}|^{\gamma} \mathcal{P}_{\lambda,\Lambda}^{-}(D^{2}\tilde{u}) \leq \tilde{f} \leq |D\tilde{u}|^{\gamma} \mathcal{P}_{\lambda,\Lambda}^{+}(D^{2}\tilde{u})
\end{align*}
with $\norm{\tilde{u}}_{L^{\infty}(B_2)} \leq 1/16$ and $\norm{\tilde{f}}_{L^{n}(B_{2})} \leq \epsilon_{1}$.
By Lemma \ref{Smallness}, we obtain $\norm{|D\tilde{u}|^{\gamma}D\tilde{u}}_{W^{1, \delta}(B_{1})} \leq C$.
Therefore, we have
\begin{align*}
    \norm{|Du|^{\gamma}Du}_{W^{1, \delta}(B_{1})} \leq Ca^{1+\gamma} \leq C \left( \norm{u}_{L^{\infty}(B_{2})}^{1+\gamma} + \norm{f}_{L^{n}(B_{2})} +\epsilon^{1+\gamma} \right),
\end{align*}
which proves Theorem \ref{main} by letting $\epsilon \rightarrow 0$.

\section{The relation between a tangent $C^{1,\alpha}$ cone and $D(|Du|^\gamma Du$)}
\label{sec3}
What we will prove in this section is the measure decay estimate of $|B_1 \setminus T_{t}|$, which implies that almost every point has a tangent $C^{1,\alpha}$ cone. Throughout this paper, we always put $\alpha := \frac{1}{1+\gamma}$.

\begin{thm} \label{tangent}
Let $\gamma \geq 0$. Assume $u \in C(\overline{B}_{2})$ satisfies 
\begin{align*}
|Du|^{\gamma} \mathcal{P}_{\lambda,\Lambda}^{-}(D^{2}u) \leq f \leq |Du|^{\gamma} \mathcal{P}_{\lambda,\Lambda}^{+}(D^{2}u) \quad\text{in}\ B_{2}
\end{align*}
in the viscosity sense.
Then there exist constants $\epsilon_{1} = \epsilon_{1}(n,\lambda, \Lambda, \gamma)>0$, $\sigma=\sigma(n,\lambda, \Lambda,\gamma)>0$ such that if $\norm{f}_{L^{n}(B_{2})} \leq \epsilon_{1}$ and $\osc_{\overline{B}_{2}} u \leq 1/8$, then for any $t>0$
$$|B_1 \setminus T_{t}| \leq Ct^{-\sigma},$$
where $C=C(n,\lambda,\Lambda,\gamma)>0$.
\end{thm}

How can we find a bound of $D(|Du|^\gamma Du)$ using a tangent $C^{1,\alpha}$ cone?
Let us consider a convex $C^{1,\alpha}$ cone $P(x)=\frac{1}{1+\alpha}|x-y_0|^{1+\alpha} + C$ of vertex $y_0 \in B_2$ and opening $1$ which touches $u$ above at $x_0$, that is, $u(x) \overset{x_0}{\leq} P(x)$.
Assume that $u$ is $C^2$ near $x_0$. If $x_0 = y_0$, then $Du(x_0) = DP(x_0) = 0$, therefore $D(|Du|^\gamma Du)(x_0) = 0$.
If $x_0 \neq y_0$, we get 
\begin{align*}
    Du(x_0) = DP(x_0), \ D^2u(x_0) \leq D^2P(x_0).
\end{align*}
Since $DP(x) = |x-y_0|^\alpha \frac{x-y_0}{|x-y_0|}$, $D^2P(x) = \frac{1}{|x-y_0|^{1-\alpha}}\left(I+(1-\alpha)\frac{x-y_0}{|x-y_0|} \otimes \frac{x-y_0}{|x-y_0|}\right)$ and $\alpha\gamma =1-\alpha$,  we have
\begin{align*}
    |DP|^\gamma D^2P(x) = I+(1-\alpha)\frac{x-y_0}{|x-y_0|} \otimes \frac{x-y_0}{|x-y_0|} \leq 2I, \quad \forall x \in \mathbb{R}^n \setminus \{y_0\}.
\end{align*}
Finally, we obtain
\begin{align*}
    |Du|^\gamma D^2u(x_0) \leq |DP|^\gamma D^2P(x_0) \leq 2I.
\end{align*}
Similarly, we can find a lower bound of $|Du|^\gamma D^2u$ when there exists a concave $C^{1,\alpha}$ cone touching $u$ from below.
Therefore, if $x_0 \in T_K(u)$ and $u$ is $C^2$ near $x_0$, we get
\begin{align*}
    -2K^{1+\gamma}I \leq |Du|^\gamma D^2u(x_0) \leq 2K^{1+\gamma}I.
\end{align*}
Moreover, we have
\begin{align*}
    |D(|Du|^\gamma Du)| &= \left|\left(I+ \gamma \frac{Du}{|Du|} \otimes \frac{Du}{|Du|}\right)(|Du|^\gamma D^2u)\right| \\
    &\leq 2(1+\gamma)K^{1+\gamma}.
\end{align*}
Now assuming that $u$ is $C^2$ and that Theorem \ref{tangent} is proved, we define $g$ as $g(x) := \frac{1}{2}\inf \{ t : x \in T_t(u) \}^{\frac{1}{1+\gamma}}$.
Then we have
\begin{align*}
    |D(|Du|^\gamma Du)(x)| \leq (1+\gamma)g(x).
\end{align*}
By Theorem \ref{tangent}, we obtain
\begin{align*}
    |B_1 \setminus T_t| = |\{ x \in B_1 : (2g(x))^{1+\gamma} \geq t\}| \leq Ct^{-\sigma}.
\end{align*}
Therefore, we conclude that $g \in L^\delta$ and $|Du|^\gamma Du \in W^{1,\delta}$ for some $\delta>0$ under the assumption that $u$ is $C^2$.

We prove the Lemma \ref{Smallness} assuming the Theorem \ref{tangent}.
Since $u$ may not be $C^2$, we will consider sup-convolution of $u$ defined by
\begin{align*}
    u^\epsilon(x_0) := \inf_{x\in B_1} \left(u(x) + \frac{1}{\epsilon}|x-x_0|^2 \right).
\end{align*}
Then $u^\epsilon$ is $C^{1,1}$ from below, $u^\epsilon \rightarrow u$ uniformly on compact subsets and $u^\epsilon$ satisfies 
\begin{align*}
    |Du^\epsilon|^{\gamma} \mathcal{P}_{\lambda,\Lambda}^{-}(D^{2}u^\epsilon) \leq f^\epsilon \quad \text{in } B_{2-\epsilon}
\end{align*}
in the viscosity sense, for some $f^\epsilon \in C(B_{2-\epsilon})$ with $f^\epsilon \rightarrow f$ uniformly on compact subsets.

We define $g^\epsilon(x) := \frac{1}{2}\inf \{ t : x \in T_t(u^\epsilon) \}^{\frac{1}{1+\gamma}}$
By Theorem \ref{tangent}, we have $|B_1 \setminus T^{-}_t(u^\epsilon)| \leq Ct^{-\sigma}$, therefore $\norm{g^\epsilon}_{L^\delta}$ is bounded.
Since $u^\epsilon$ is $C^{1,1}$ a.e., by using the above argument we have
$$-g^\epsilon I \leq |Du^\epsilon|^\gamma D^2u^\epsilon \quad \text{ a.e. }$$
Let $e_i(x)$ be an eigenvalue of $|Du^\epsilon|^\gamma D^2u^\epsilon(x)$.
Then the above inequality implies $-g^\epsilon \leq e_i$.
Moreover, since $\mathcal{P}^{-}(|Du^\epsilon|^{\gamma}D^2u^\epsilon) \leq f^\epsilon$,
we have $\Lambda \sum_{e_i <0}e_i + \lambda \sum_{e_i>0}e_i \leq f^\epsilon$.
Therefore, $\lambda e_i \leq \lambda\sum_{e_i>0}e_i \leq f^\epsilon -\Lambda \sum_{e_i <0}e_i \leq f^\epsilon +\Lambda ng^\epsilon$.
Thus, we have
\begin{align*}
    -g^\epsilon I \leq |Du^\epsilon|^\gamma D^2u^\epsilon \leq \left(\frac{f^\epsilon}{\lambda} + \frac{\Lambda}{\lambda}ng^\epsilon\right)I \quad \text{ a.e. }
\end{align*}
For nonsymmetric matrix $M$, we write $M_{\text{sym}} := \frac{1}{2}(M+M^T)$.
Since $D(|Du|^\gamma Du)$ may be nonsymmetric, we will instead use $D(|Du|^\gamma Du)_{\text{sym}}$.
Notice that
\begin{align*}
    D(|Du^\epsilon|^\gamma Du^\epsilon)_{\text{sym}} &= |Du^\epsilon|^\gamma D^2u^\epsilon + \frac{\gamma}{2}\frac{Du^\epsilon}{|Du^\epsilon|} \otimes \frac{Du^\epsilon}{|Du^\epsilon|} |Du^\epsilon|^\gamma D^2u^\epsilon + \frac{\gamma}{2}|Du^\epsilon|^\gamma D^2u^\epsilon \frac{Du^\epsilon}{|Du^\epsilon|} \otimes \frac{Du^\epsilon}{|Du^\epsilon|} \\
    &=(I+B)A(I+B)-BAB,
\end{align*}
where $A = |Du^\epsilon|^\gamma D^2u^\epsilon$ and $B=\frac{\gamma}{2}\frac{Du^\epsilon}{|Du^\epsilon|} \otimes \frac{Du^\epsilon}{|Du^\epsilon|}$.
Since $0\leq B \leq \frac{\gamma}{2}I$, we have $(I+B)A(I+B) \geq -\left(1+\frac{\gamma}{2}\right)^2g^\epsilon I$ and $BAB \leq \frac{\gamma^2}{4}\left(\frac{f^\epsilon}{\lambda} + \frac{\Lambda}{\lambda}ng^\epsilon\right)I$.
Therefore, we obtain
\begin{align*}
    -\tilde{g}^\epsilon I=-\left(\left(1+\frac{\gamma}{2}\right)^2g^\epsilon+\frac{\gamma^2}{4}\left(\frac{f^\epsilon}{\lambda} + \frac{\Lambda}{\lambda}ng^\epsilon\right)\right)I\leq D(|Du^\epsilon|^\gamma Du^\epsilon)_{\text{sym}}
\end{align*}
where $\tilde{g}^\epsilon=\left(1+\frac{\gamma}{2}\right)^2g^\epsilon+\frac{\gamma^2}{4}\left(\frac{f^\epsilon}{\lambda} + \frac{\Lambda}{\lambda}ng^\epsilon\right)$.
Notice that $\norm{\tilde{g}^\epsilon}_{L^\delta}$ is also bounded.
Thus, for any constant vector $e \in \mathbb{R}^n$ with $|e| = 1$, we have $e^T(-\tilde{g}^\epsilon I)e \leq e^TD(|Du^\epsilon|^\gamma Du^\epsilon)_{\text{sym}}e$. For any $0 \leq \phi \in C^{\infty}_0(B_1)$, we get
\begin{align*}
    \int_{B_1} e^T D(|Du^\epsilon|^\gamma Du^\epsilon)_{\text{sym}}e \phi \, dx &= \int_{B_1} \sum_{i,j} D_i(|Du^\epsilon|^\gamma D_ju^\epsilon)e_i e_j \phi \, dx \\
    &=- \int_{B_1} \sum_{i,j} (|Du^\epsilon|^\gamma D_ju^\epsilon)D_i\phi e_i e_j \, dx \\
    &=- \int_{B_1} e^T D\phi \otimes (|Du^\epsilon|^\gamma Du^\epsilon) e \, dx.
\end{align*}
Therefore, we obtain that for any $0\leq \phi\in C^{\infty}_0 (B_1)$ and $e \in \mathbb{R}^n$ with $|e| = 1$,
\begin{align} \label{weak}
    \int_{B_1} -\tilde{g}^\epsilon\phi \, dx \leq - \int_{B_1} e^T D\phi \otimes (|Du^\epsilon|^\gamma Du^\epsilon) e \, dx.
\end{align}
Note that for any $x_0 \in T^{-}_t(u^\epsilon)$, there exists a $C^{1,\alpha}$ cone $P$ with vertex $y_0 \in B_{2}$ and opening $t$ which touches $u^\epsilon$ at $x_0$.
Since $u^\epsilon$ is $C^{1,1}$ and $P$ is $C^1$, we obtain $|Du^\epsilon(x_0)| = |DP(x_0)| = t|x_0 - y_0|^\alpha  \leq 3t$.
Therefore, we get $||Du^\epsilon|^\gamma Du^\epsilon(x)| \leq 3g^\epsilon(x)$ and $$\norm{|Du^\epsilon|^\gamma Du^\epsilon}_{L^\delta(B_1)} \leq C,$$
for some constant $C(n,\lambda,\Lambda, \gamma)>0$. 
Letting $\epsilon \rightarrow 0$ at \eqref{weak}, we conclude that 
$$\int_{B_1} -\tilde{g}\phi \, dx \leq - \int_{B_1} e^T D\phi \otimes(|Du|^\gamma Du) e \, dx,$$
which implies $-\tilde{g}I \leq D(|Du|^\gamma Du)_{\text{sym}}$ in weak sense.
Using the same method as for the supersolution case, we obtain that $-\tilde{g}I \leq D(|Du|^\gamma Du)_{\text{sym}} \leq \tilde{g}I$ in weak sense, and hence
$$\norm{|Du|^{\gamma}Du}_{W^{1, \delta}(B_{1})} \leq \norm{\tilde{g}}_{L^\delta(B_1)} \leq C.$$

\section{Proof of Theorem \ref{tangent}}
\label{sec4}
In this section, we give a proof of Theorem \ref{tangent}.
To prove Theorem \ref{tangent}, we need the key lemma, Lemma \ref{Key}, which implies that if there is a small measure of tangent points of opening $1$ and if the maximal function of $f$ is small enough, then there is a large measure of tangent points of opening $M>1$.
The main strategy of our proof is based on \cite{Li18}, \cite{Savin07} and \cite{Colombo14}.
However, since the growth and degeneracy of a $C^{1,\alpha}$ cone are different from those of a paraboloid, the proof is divided into three cases according to the distance of the vertex of $C^{1,\alpha}$ cone.
Note that the main difficulty of Steps 1 and 2 in Case 1 comes from the growth of a $C^{1,\alpha}$ cone, and the difficulty of Step 3 comes from degeneracy of the equation.

\begin{lem} \label{Key}
Let $B_{2} \subset \Omega$ with $\Omega$ is convex, $\gamma \geq 0$ and $0<\theta_{0}<1$.
Assume that $u \in C(\Omega)$ satisfies
\begin{align} \label{subPDE}
|Du|^{\gamma} \mathcal{P}_{\lambda,\Lambda}^{-}(D^{2}u) \leq f \quad\text{in}\ B_{2}
\end{align}
in the viscosity sense.
Then there exist constants $\epsilon_{2} = \epsilon_{2}(n,\lambda, \Lambda, \theta_{0})>0$, $M=M(n,\lambda, \Lambda)>1$, $0< \Theta = \Theta(n,\lambda, \Lambda)<1$ and $0<\theta = \theta(n,\lambda, \Lambda, \theta_{0}) < \min{(\theta_{0}, \Theta)}<1$ such that if 
$$|B_{1} \cap T^{-}_{1} \cap \{x\in \Omega : \mathcal{M}(|f|^n)(x) \leq \epsilon_2\}| \geq \theta |B_{1}|,$$ then 
$$|B_{1} \cap T^{-}_{M}| \geq \Theta |B_{1}|.$$
\end{lem}

\begin{proof}
We assume that $u$ is semi-concave in $B_2$.
In general case, we can regularize $u$ to be semi-concave by using inf-convolution (See \cite{Li17, Caffarelli95}).
Let $0<\theta<1$ which will be chosen later.
Since $|B_{1} \cap T^{-}_{1}\cap \{x\in \Omega : \mathcal{M}(|f|^n)(x) \leq \epsilon_2\}| \geq \theta |B_{1}|$, there exists $\kappa=\kappa(n,\theta) \in (0,1)$ such that
$$B_{1-\kappa} \cap T^{-}_{1} \cap \{x\in \Omega : \mathcal{M}(|f|^n)(x) \leq \epsilon_2\} \neq \emptyset. $$
Let $x_{1} \in B_{1-\kappa} \cap T^{-}_{1} \cap \{x\in \Omega : \mathcal{M}(|f|^n)(x) \leq \epsilon_2\}$.
Then there exist a vertex $y_{1} \in \overline{\Omega}$ and a $C^{1,\alpha}$ cone $P_{1}$ such that
\begin{align} \label{P_1}
    P_{1}(x) = -\frac{1}{1+\alpha}|x-y_{1}|^{1+\alpha} + C \overset{x_{1}}{\leq} u \quad \text{ in } \Omega.
\end{align}
We divide the proof into the three cases according to the distance between $y_1$ and the origin.
Case 1 is $|y_{1}| \geq C_{0}$ for large $C_{0}$, Case 2 is $ 2 \leq |y_{1}| < C_{0}$, Case 3 is $|y_{1}| < 2$.
For each case, the proof is divided into several steps as in \cite{Li18}.

\textbf{Case 1.} $|y_{1}| \geq C_{0}$ for large $C_{0}$ to be determined later in \textbf{Step 2}.

\textbf{Step 1.} We prove that  there exist $x_2 \in B_{1/2}$ and $C_{1}=C_{1}(n,\lambda,\Lambda,\gamma)>1$ independent of $C_{0}$ such that
\begin{align} \label{Case1Step1}
    (u-P_{1})(x_2) \leq \frac{C_{1}}{|y_1|^{1-\alpha}}.
\end{align}
Consider the barrier function $\phi(x) := (|x|^{-p} - (3/4)^{-p})/p$ for $|x| > 1/4$ with large $p>0$ to be determined later and extend $\phi$ smoothly.
For each vertex point $y_V \in B_{1/4} \cap B_{3/4}(x_{1})$, we define the test function 
\begin{align*}
    \psi(x) := P_{1}(x) + \frac{1}{|y_1|^{1-\alpha}}\phi(x-y_V).
\end{align*}
Let $x_T \in \overline{B}_{3/4}(y_V)$ be a touching point such that
\begin{align} \label{Case1vis1}
    (u-\psi)(x_T)=\min_{\overline{B}_{3/4}(y_V)}(u-\psi).
\end{align}
Then we get $x_T \in B_{3/4}(y_V)$ since
\begin{align*}
    & (u-\psi)|_{\partial B_{3/4}(y_V)} \geq 0, \\
    & (u-\psi)(x_T) \leq (u-\psi)(x_1) = -\frac{1}{|y_0|^{1-\alpha}}\phi(x_1) <0.
\end{align*}
\textbf{Claim}: For small enough $\epsilon_2=\epsilon_2(n,\lambda,\Lambda,\gamma,\kappa)>0$, there exists a vertex point $y^{0}_V \in B_{1/4} \cap B_{3/4}(x_{1})$, such that a corresponding touching point $x^{0}_T$ is very close to vertex point, that is, $x^{0}_T \in B_{1/4}(y^0_V)$.

If we prove this claim, we get $x^{0}_T \in B_{1/2}$ since
$$|x^{0}_T| \leq |x^{0}_T-y^{0}_V| + |y^{0}_V| < \frac{1}{4} + \frac{1}{4} = \frac{1}{2}.$$
Hence we show \eqref{Case1Step1} by $x_2=x^{0}_T$ and $C_1 = \norm{\phi}_{L^{\infty}}$ since
$$u(x^{0}_T) < \phi(x^{0}_T) \leq P_{1}(x^{0}_T) + \frac{1}{|y_0|^{1-\alpha}}\norm{\phi}_{L^{\infty}}.$$
Suppose by contradiction that for all $y_V \in B_{1/4} \cap B_{3/4}(x_{1})$, the corresponding touching points are always in $x_T \in B_{3/4}(y_V) \setminus B_{1/4}(y_V)$. Since $|y_V| \leq 1/4$, we have
\begin{align} \label{Case1Propx_T}
    \frac{1}{4} <|x_T-y_V| < \frac{3}{4}, \quad |x_T| < 1.
\end{align}
We define $T$ by the set of corresponding touching points $x_T$, then $T \subset B_1$.

The proof of claim is divided into four minor steps.

\textbf{Step 1-1.}
From \eqref{Case1vis1}, we see that
$$ \psi + \min_{\overline{B}_{3/4}(y_V)}(u-\psi) \overset{x_T}{\leq} u \quad \text{ in } \overline{B}_{3/4}(y_V). $$
By the definition of viscosity solution, we obtain
\begin{align*}
    |D\psi(x_T)|^{\gamma} \mathcal{P}_{\lambda,\Lambda}^{-}(D^{2}\psi(x_T)) \leq f(x_T).
\end{align*}
A direct computation yields
\begin{align*}
    |D\psi(x_T)| = \left|-|x_T - y_1|^{\alpha} \frac{x_T - y_1}{|x_T - y_1|} -\frac{1}{|y_1|^{1-\alpha}|x_T-y_V|^{p+1}} \frac{x_T-y_V}{|x_T-y_V|}\right|.
\end{align*}
Note that since $|x_T| < 1$ and $|y_1|>C_0$, we have $|x_T - y_1|^{\alpha} > \left(1-\frac{1}{C_0}\right)|y_1|^{\alpha}$.
From \eqref{Case1Propx_T}, for large enough $C_{0}> 2(4^{p+1}+1)$, we get
\begin{align} \label{Case1Dpsi}
    |D\psi(x_T)| >  \left( 1-\frac{1}{C_0} \right) {|y_1|^{\alpha}} - \frac{4^{p+1}}{|y_1|^{1-\alpha}} > \frac{1}{2}|y_1|^{\alpha}.
\end{align}

A direct computation with \eqref{Case1Propx_T} also shows
\begin{align*}
\mathcal{P}^{-}(D^{2}\psi(x_T)) &\geq \mathcal{P}^{-}(D^{2}P_{1}(x_T)) + \frac{1}{|y_1|^{1-\alpha}}{P}^{-}(D^{2}\phi(x_T-y_V)) \\
 &= - \frac{\Lambda(n-1+\alpha)}{|x_T - y_1|^{1-\alpha}} +\frac{1}{|y_1|^{1-\alpha}|x_T-y_V|^{p+2}} ((p+1)\lambda - (n-1)\Lambda) \\
 &\geq \frac{1}{|y_1|^{1-\alpha}} \left(-2\Lambda(n-1+\alpha) + \left(\frac{4}{3}\right)^{p+2}((p+1)\lambda - (n-1)\Lambda)\right) \\
 &\geq \frac{2^\gamma}{|y_1|^{1-\alpha}}
\end{align*}
for large enough $p=p(n,\lambda,\Lambda,\gamma)>0$.

Since $\alpha \gamma = 1-\alpha$, we get
\begin{align*}
     1 \leq |D\psi(x_T)|^{\gamma} \mathcal{P}_{\lambda,\Lambda}^{-}(D^{2}\psi(x_T)) \leq f(x_T), \quad \forall x_T \in T.
\end{align*}
Finally, we want to show that $Du$ is Lipschitz on the set of a touching point $T$.
Recalling that $D^2\psi$ is bounded, it can be proved by using the same method in \cite{Li18}, and we will prove similar statement in \textbf{Step 3}. Thus we omit the proof.
Therefore, by the Radamacher theorem, $D^2u$ exists almost everywhere on $T$.

\textbf{Step 1-2.} For any $x_T \in T$, we have
\begin{align*}
    Du(x_T)= DP_1(x_T) + \frac{1}{|y_1|^{1-\alpha}}D\phi(x_T-y_V), \\
    D^{2}u(x_T) \geq D^{2}P_1(x_T) + \frac{1}{|y_1|^{1-\alpha}}D^{2}\phi(x_T-y_V).
\end{align*}
Let $v = u-P_1$. Since $Dv(x_T) = \frac{1}{|y_1|^{1-\alpha}}D\phi(x_T-y_V)$, we get
\begin{align*}
    y_V &= x_T - (D\phi)^{-1}(|y_1|^{1-\alpha} Dv(x_T)).
\end{align*}
Therefore, for each touching point $x_T$, the corresponding vertex $y_V$ is unique. By differentiating with respect to $x_T$, we have
\begin{align*}
    D_{x_T}y_V = I - D((D\phi)^{-1})(|y_1|^{1-\alpha}Dv(x_T)) \cdot |y_1|^{1-\alpha} D^2 v(x_T).
\end{align*}
By the inverse function theorem, we have $D((D\phi)^{-1}) \circ D\phi(x_T -y_V) = (D^2 \phi)^{-1}(x_T -y_V)$.
Combining this with $|y_1|^{1-\alpha}Dv(x_T) = D\phi(x_T-y_V)$, we get
\begin{align*}
    D_{x_T}y_V &= I - (D^2 \phi)^{-1}(x_T -y_V) \cdot |y_1|^{1-\alpha} D^2 v(x_T) \\
    &= (D^2 \phi)^{-1}(x_T -y_V) \cdot (D^2 \phi (x_T -y_V) - |y_1|^{1-\alpha} D^2 v(x_T)) \\
    &= D((D\phi)^{-1}) (|y_1|^{1-\alpha}Dv(x_T)) \cdot (D^2 \phi (x_T -y_V) - |y_1|^{1-\alpha}D^2 v(x_T)).
\end{align*}
A direct calculation shows that
\begin{align*}
    (D\phi)^{-1}(x) &= -\frac{1}{|x|^{\frac{1}{p+1}}}\frac{x}{|x|}, \\
    D((D\phi)^{-1})(x) &= -\frac{1}{|x|^{1+\frac{1}{p+1}}}\left(I-\left(1+\frac{1}{p+1}\right)\frac{x}{|x|} \otimes \frac{x}{|x|}\right).
\end{align*}
Therefore, we obtain
\begin{align*}
    D_{x_T}y_V = \frac{1}{|y_1|^{\frac{1-\alpha}{p+1}}{|Dv|^{1+\frac{1}{p+1}}}} \left(I-\left(1+\frac{1}{p+1}\right) \frac{Dv}{|Dv|} \otimes \frac{Dv}{|Dv|}\right)\left(D^{2}v(x_T)-\frac{1}{|y_1|^{1-\alpha}}D^{2}\phi(x_T-y_V)\right).
\end{align*}
Note that $|Dv(x_T)| \neq 0$ since it follows from \eqref{Case1Propx_T} that
\begin{align} \label{Case1Dv}
    |Dv(x_T)| = \frac{1}{|y_1|^{1-\alpha}}|D\phi(x_T-y_V)| = \frac{1}{|y_1|^{1-\alpha}|x_T-y_V|^{p+1}} 
    \geq \left( \frac{4}{3} \right)^{p+1}\frac{1}{|y_1|^{1-\alpha}} > 0.
\end{align}
Consequently, since $\det\left(I-\left(1+\frac{1}{p+1}\right) \frac{Dv}{|Dv|}\otimes \frac{Dv}{|Dv|}\right) = -\frac{1}{p+1}$, we have
\begin{align*}
    \det D_{x_T}y_V = -\frac{1}{p+1}\left(\frac{1}{|Dv|^{1+\frac{1}{p+1}}|y_1|^{\frac{1-\alpha}{p+1}}}\right)^{n} \det\left(D^{2}v(x_T)-\frac{1}{|y_1|^{1-\alpha}}D^{2}\phi(x_T-y_V)\right).
\end{align*}

\textbf{Step 1-3.} We want to find a bound for $\det\left(D^{2}v(x_T)-\frac{1}{|y_1|^{1-\alpha}}D^{2}\phi(x_T-y_V)\right)$ by using \eqref{subPDE}.
By properties of the Pucci operators, we see that
\begin{align*}
    |Du|^{\gamma} \mathcal{P}^{-}(D^{2}u - D^{2}\psi) +|Du|^{\gamma} \mathcal{P}^{-}(D^{2}\psi) \leq |Du|^{\gamma} \mathcal{P}^{-}(D^{2}u) \leq f.
\end{align*}
From \textbf{Step 1-1}, we have $Du(x_T) = D\psi(x_T)$, $|D\psi(x_T)|^{\gamma} \mathcal{P}^{-}(D^{2}\psi(x_T)) \geq 1$. Thus,
\begin{align*}
    |Du(x_T)|^{\gamma} \mathcal{P}^{-}(D^{2}u - D^{2}\psi)(x_T) \leq f(x_T).
\end{align*}
Since $D^{2}u(x_T) - D^{2}\psi(x_T) = D^{2}v(x_T)-\frac{1}{|y_1|^{1-\alpha}}D^{2}\phi(x_T-y_V) \geq 0$, we get
\begin{align*}
    \lambda \tr(D^{2}u(x_T) - D^{2}\psi(x_T)) = \mathcal{P}^{-}(D^{2}u(x_T) - D^{2}\psi(x_T)) \leq \frac{f(x_T)}{|Du(x_T)|^{\gamma}}
\end{align*}
and using the fact that $\det A \leq \left(\frac{\tr A}{n}\right)^n$ for any positive semi-definite matrix $A \geq 0$, we deduce that
\begin{align*}
    \det\left(D^{2}v(x_T)-\frac{1}{|y_1|^{1-\alpha}}D^{2}\phi(x_T-y_V)\right) \leq \left(\frac{\tr(D^{2}u(x_T) - D^{2}\psi(x_T))}{n}\right)^{n} \leq C\left(\frac{f(x_T)}{|Du(x_T)|^{\gamma}}\right)^{n}.
\end{align*}
From \eqref{Case1Dpsi} and \eqref{Case1Dv}, we have $ |Du(x_T)| = |D\psi(x_T)| \geq \frac{1}{2}|y_1|^{\alpha}$ and $|Dv(x_T)| \geq \frac{C}{|y_1|^{1-\alpha}}$. Thus, we conclude that
\begin{align*}
    |\det D_{x_T}y_V| &\leq C \left|\frac{f(x_T)}{|y_1|^{\frac{1-\alpha}{p+1}}|Dv(x_T)|^{1+\frac{1}{p+1}}|Du(x_T)|^{\gamma}}\right|^{n}\\
    &\leq C\left|\frac{f(x_T)}{|y_1|^{\frac{1-\alpha}{p+1}}|y_1|^{-(1-\alpha)(1+\frac{1}{p+1})}|y_1|^{\alpha \gamma}}\right|^{n}\\
    &= C|f(x_T)|^{n}.
\end{align*}

\textbf{Step 1-4.} Consider the mapping $m : T \rightarrow B_{1/4} \cap B_{3/4}(x_{1})$ defined by $m(x_T) =y_V$. Then by the area formula,
\begin{align*}
    |B_{1/4} \cap B_{3/4}(x_{1})| \leq \int_{T}|\det D_{x_T}y_V| dx_T \leq C \int_{T} |f(x_T)|^{n} dx_T.
\end{align*}
Since $B_{1/4} \cap B_{3/4}(x_{1})$ contains a ball of radius $\kappa/2$ and $T \subset B_1 \subset B_2(x_1)$, we conclude that
\begin{align*}
    \left( \frac{\kappa}{2} \right)^{n}|B_1| &\leq C \int_{B_2(x_1)} |f(x_T)|^{n} dx_T\\
    &\leq C\mathcal{M}(|f|^n)(x_1)|B_2(x_1)| \\
    &\leq C\epsilon_2 2^n |B_1|,
\end{align*}
which makes contradiction for small enough $\epsilon_{2} < 4^{-n}C^{-1}\kappa^n$.
This proves the claim and therefore \textbf{Step 1}.\\

\textbf{Step 2.} We want to find universal constants $M_1=M_1(n,\lambda,\Lambda, \gamma)>1$, $C_0=C_0(n,\lambda,\Lambda, \gamma)>2$, and a set $V_1 \subset \Omega$ with universal positive measure such that $T^{-}_{M_1}(V_1) \subset B_1$.

For each $\tilde{x} \in T^{-}_{M_1}(V_1)$, there exists $\tilde{y} \in V_1$ such that
\begin{align} \label{Case1P_M}
    P_{M_1}(x) = -\frac{M_1}{1+\alpha}|x-\tilde{y}|^{1+\alpha} + C \overset{\tilde{x}}{\leq} u \quad \text{ in } \Omega.
\end{align}
Consider the function $Q$ defined by the difference of two $C^{1,\alpha}$ cones,
\begin{align*}
    Q(x)& := P_{M_1}(x)- P_1(x)\\
    &= -\frac{M_1}{1+\alpha}|x-\tilde{y}|^{1+\alpha} +-\frac{1}{1+\alpha}|x-y_1|^{1+\alpha} +C'.
\end{align*}
By using \eqref{Case1Step1}, \eqref{P_1} and \eqref{Case1P_M}, we get
\begin{align*}
    Q(\tilde{x}) &=P_{M_1}(\tilde{x})- P_1(\tilde{x}) = u(\tilde{x}) - P_1(\tilde{x}) \geq 0, \\
    Q(x_2) & = P_{M_1}(x_2)- P_1(x_2)\leq P_{M_1}(x_2) -u(x_2) + \frac{C_{1}}{|y_1|^{1-\alpha}} \leq \frac{C_{1}}{|y_1|^{1-\alpha}}.
\end{align*}
Therefore, we deduce that
\begin{align} \label{Case1Q}
    Q(x_2) \leq Q(\tilde{x}) + \frac{C_{1}}{|y_1|^{1-\alpha}}.
\end{align}
Using the fact that $x_2 \in B_{1/2}$, we want to show $\tilde{x} \in B_1$ which implies $T^{-}_{M_1}(V_1) \subset B_1$.

Heuristically, the idea is consider $Q$ as cone-like and choose $V_1$ such that vertex of $Q$ is close to $x_2$.
Then we may show that $\tilde{x}$ is also close to vertex of $Q$ by using \eqref{Case1Q} and therefore we conclude that $\tilde{x} \in B_1$ by using $x_2 \in B_{1/2}$ and closeness of vertex and $x_2$.

Hence, we first find the vertex of $Q$.
It seems reasonable that we call a point vertex of $Q$ if $DQ$ is 0 only at that point.
A direct computation shows
\begin{align*}
    DQ(x) = -M_1|x-\tilde{y}|^{\alpha}\frac{x-\tilde{y}}{|x-\tilde{y}|} + |x-y_1|^{\alpha}\frac{x-y_1}{|x-y_1|}.
\end{align*}
If $DQ(x)=0$, then $x-\tilde{y}$ and $x-y_1$ is parallel, which implies that a vertex should be on the line $\overleftrightarrow{\tilde{y} y_1}$.
Therefore, a vertex $y_0$ which satisfies $DQ(y_0)=0$ is only one point,
\begin{align*}
    y_0 = \frac{M_1^{1/\alpha} \tilde{y} - y_1}{M_1^{1/\alpha}-1}.
\end{align*}
Therefore, we choose $V_1 \subset \mathbb{R}^n$ as
$$V_1 = \overline{B}_{\frac{\sqrt{\alpha}}{8}\frac{M_1^{1/\alpha}-1}{M_1^{1/\alpha}}} \left( \frac{1}{M_1^{1/\alpha}}y_1 + \frac{M_1^{1/\alpha}-1}{M_1^{1/\alpha}}x_2 \right).$$ 
Then $y_0 \in B_{\sqrt{\alpha} / 8}(x_2)$ if $\tilde{y} \in V_1$. 
Moreover, since $y_1 \in \overline{\Omega}$, $B_{\sqrt{\alpha}/8}(x_2) \subset B_1 \subset \Omega$, and $\Omega$ is convex, we have $V_1 \subset \Omega$.

\textbf{Claim}: For large $M_1=M_1(n,\lambda,\Lambda, \gamma)>1$ and $C_0=C_0(n,\lambda,\Lambda, \gamma)>2$, we have $\tilde{x} \in B_{1/4}(y_0)$.

Once we have proved the claim, we conclude that
$$ |\tilde{x}| \leq |\tilde{x} -y_0| + |y_0 - x_2| +|x_2| \leq \frac{1}{4} + \frac{\sqrt{\alpha}}{8} + \frac{1}{2} \leq \frac{7}{8}.$$
Therefore, we have $\tilde{x} \in B_1$ which proves $T^{-}_{M_1}(V_1) \subset B_1$.

To do this, we assume $y_0=0$ by translation. Then we only need to prove $|\tilde{x}| \leq \frac{1}{4}$. We have $|x_2| < \frac{\sqrt{\alpha}}{8}$ and $\tilde{y} = \frac{1}{M_1^{1/{\alpha}}} y_1$. Note that $Q$ is changed into
\begin{align*}
    Q(x) = -\frac{M_1}{1+\alpha}\left|x-\frac{1}{M_1^{1/{\alpha}}} y_1 \right|^{1+\alpha} +\frac{1}{1+\alpha}|x-y_1|^{1+\alpha} + C'.
\end{align*}
Define $a_Q$ to be
\begin{align*}
    a_Q &:= \sup_{x \in \mathbb{R}^n}\left\{ |x| : Q(x) \geq \min_{B_{\sqrt{\alpha} / 8}} Q - \frac{C_{1}}{|y_1|^{1-\alpha}}\right\}.
\end{align*}
Note that $a_Q = a_{Q+C}$ for any constant $C$. Thus we assume $Q(0) =0$ by adding some constant. Observe that $|\tilde{x}| \leq a_Q$ since $|x_2| < \sqrt{\alpha} / 8$ and
\begin{align*}
    Q(\tilde{x}) \geq Q(x_2) - \frac{C_{1}}{|y_1|^{1-\alpha}}
    \geq \min_{B_{\sqrt{\alpha} / 8}} Q - \frac{C_{1}}{|y_1|^{1-\alpha}}.
\end{align*}
Therefore, we only need to show $a_Q \leq 1/4$. However, it is not easy to find $a_Q$ since we do not know what $Q$ looks like.
Instead, if we choose sufficiently large $C_0$, we can approximate the Hessian of $Q$, and then approximate $Q$ by an elliptic paraboloid in $B_1$.
A direct computation shows
\begin{align*}
    D^2Q(x) & = -\frac{M_1}{|x-\frac{1}{M_1^{1/{\alpha}}} y_1|^{1-\alpha}}\left(I + (\alpha -1)\frac{x-\frac{1}{M_1^{1/{\alpha}}} y_1}{|x-\frac{1}{M_1^{1/{\alpha}}} y_1|} \otimes \frac{x-\frac{1}{M_1^{1/{\alpha}}} y_1}{|x-\frac{1}{M_1^{1/{\alpha}}} y_1|}\right) \\
    & \quad +\frac{1}{|x-y_1|^{1-\alpha}}\left(I + (\alpha -1)\frac{x-y_1}{|x-y_1|} \otimes \frac{x-y_1}{|x-y_1|}\right).
\end{align*}
Since $|x| < 1$, the direction of $\frac{x-y_1}{|x-y_1|}$ can be approximated to the direction of $-\frac{y_1}{|y_1|}$ for large $|y_1|>C_0$, and $\left(1-\frac{1}{C_0}\right)^{1-\alpha}|y_1|^{1-\alpha}<|x-y_1|^{1-\alpha}< \left(1+\frac{1}{C_0}\right)^{1-\alpha}|y_1|^{1-\alpha}$. 
Therefore, for large $C_0=C_0(n,\gamma,\lambda,\Lambda, M_1)$,
\begin{align*}
    D^2Q(x) &\underset{C_0}{\approx} -\left(\frac{M_1}{|\frac{1}{M_1^{1/{\alpha}}} y_1|^{1-\alpha}} - \frac{1}{|y_1|^{1-\alpha}}\right) \left(I+(\alpha -1) \frac{y_1}{|y_1|} \otimes \frac{y_1}{|y_1|}\right)\\
    &= -\frac{M_1^{1/\alpha}-1}{|y_1|^{1-\alpha}}\left(I+(\alpha -1) \frac{y_1}{|y_1|} \otimes \frac{y_1}{|y_1|}\right) \quad\text{in}\ B_{1}.
\end{align*}
Here, $A \underset{C_0}{\approx} B$ means there exists $\epsilon=\epsilon(C_0)>0$ such that $\lim_{C_0 \rightarrow \infty} \epsilon (C_0)=0$ and $(1-\epsilon)A \leq B \leq (1+\epsilon)A$. 
Since $Q(0)=0$ and $DQ(0)=0$, $Q$ can be approximated by the elliptic paraboloid with opening degeneracy $1/|y_1|^{1-\alpha}$ in $B_1$,
\begin{align*}
    Q(x) \underset{C_0}{\approx}-\frac{M_1^{1/\alpha}-1}{2|y_1|^{1-\alpha}} \left( |x'|^{2} + \alpha|x_{y_1}|^{2} \right) \quad\text{in}\ B_{1},
\end{align*}
where $x_{y_1}$ is a coordinate parallel to $y_1$ and $x'$ is the other coordinates perpendicular to $y_1$.

Therefore, we can compute $\min_{B_{\frac{\sqrt{\alpha}}{8}}} Q$ as
\begin{align*}
    \min_{B_{\frac{\sqrt{\alpha}}{8}}} Q \geq -\frac{M_1^{1/\alpha}-1}{2|y_1|^{1-\alpha}}\left(\frac{\sqrt{\alpha}}{8}\right)^2.
\end{align*}
Consequently, we have
\begin{align*}
    \min_{B_{\frac{\sqrt{\alpha}}{8}}} Q - \frac{C_{1}}{|y_1|^{1-\alpha}} &\geq -\frac{M_1^{1/\alpha}-1}{2|y_1|^{1-\alpha}}\left(\frac{\sqrt{\alpha}}{8}\right)^2 - \frac{C_{1}}{|y_1|^{1-\alpha}}\\
    &\geq -\frac{M_1^{1/\alpha}-1}{2|y_1|^{1-\alpha}}\left(\frac{\sqrt{\alpha}}{4}\right)^2,
\end{align*}
where we choose large $M_1=M_1(n, \gamma, C_1)$ so that $\frac{M_1^{1/\alpha}-1}{2}\left(\left(\frac{\sqrt{\alpha}}{4}\right)^2 - \left(\frac{\sqrt{\alpha}}{8}\right)^2\right) \geq C_1$.
Therefore,
\begin{align*}
    a_Q &= \sup\left\{ |x| : Q(x) \underset{C_0}{\approx} -\frac{M_1^{1/\alpha}-1}{2|y_1|^{1-\alpha}} \left( |x'|^{2} + \alpha|x_{y_1}|^{2} \right) \geq \min_{B_{\frac{\sqrt{\alpha}}{8}}} Q - \frac{C_{1}}{|y_1|^{1-\alpha}}\right\} \\
    &\leq \sup\left\{ |x| : -\frac{M_1^{1/\alpha}-1}{2|y_1|^{1-\alpha}} \left( |x'|^{2} + \alpha|x_{y_1}|^{2} \right) \geq -\frac{M_1^{1/\alpha}-1}{2|y_1|^{1-\alpha}}\left(\frac{\sqrt{\alpha}}{4}\right)^2 \right\} =\frac{1}{4}.
\end{align*}
By using this method, we want to say $a_Q \leq \frac{1}{4}$.
However, we can approximate $Q$ only in $B_1$ so there is a possibility that $Q$ increases outside of $B_1$ and $a_Q \gg 1$.

\textbf{Claim 2}: For any open set $U\subset \mathbb{R}^n$ with the vertex $0 \in U$, 
\begin{align} \label{cone}
    \max_{\mathbb{R}^{n} \setminus U} Q = \max_{\partial U} Q.
\end{align}

Once we have proved \textbf{Claim 2}, then we can consider $Q$ only in $B_1$  since $\max_{\mathbb{R}^{n} \setminus B_1} Q = \max_{\partial B_1} Q$. Therefore $a_Q \leq 1/4$, which proves the claim and therefore $T^{-}_{M_1}(V) \subset B_1$.

To prove \textbf{Claim 2}, we first recall that
\begin{align} \label{cone2}
    DQ=0 \text{ at only one point } 0, \quad \lim_{|x| \rightarrow \infty} Q(x) = -\infty.
\end{align}
Since it is trivial that $\max_{\mathbb{R}^{n} \setminus U} Q \geq \max_{\partial U} Q$, we prove only $\max_{\mathbb{R}^{n} \setminus U} Q \leq \max_{\partial U} Q$.
Suppose by contradiction that there exists a point $x_0 \in \mathbb{R}^{n} \setminus \overline{U}$ such that $Q(x_0) > \max_{\partial U} Q$.
Then consider the set 
$$A := \left\lbrace x \in \mathbb{R}^{n} \setminus \overline{U} : Q(x) > \max_{\partial U} Q \right\rbrace. $$ 
Note that $A$ contains $x_0$ and open.
Since $\lim_{|x| \rightarrow \infty} Q(x) = -\infty$, we know that $A$ is bounded.
Thus, $\overline{A} \subset \mathbb{R}^{n} \setminus {U}$ is also bounded and closed.
By the extreme value theorem, there exists $y_0 \in \overline{A}$ such that $Q(y_0) = \max_{y \in \overline{A}}Q(y) $.
However, since $Q(y_0) \geq Q(x_0) > \max_{\partial U} Q$, we know $y_0 \notin \partial U$ and $y_0 \notin \{ x \in \mathbb{R}^{n} \setminus \overline{U} : Q(x) = \max_{\partial U} Q \}$.
Consequently, $y_0 \in A$ and $DQ(y_0) =0$, but $DQ = 0$ at only one point $0\in U$ and $y_0 \notin U$, which makes contradiction.
This proves \textbf{Claim 2} and finishes \textbf{Step 2}.

\textbf{Case 2.} $2 \leq |y_{1}| < C_{0}$

\textbf{Step 1.} As in \textbf{Case 1}, we show that there exist $x_2 \in B_{1/2}$ and $C_2=C_2(n,\lambda,\Lambda)$ such that
\begin{align} \label{Case2Step1}
    (u-P_1)(x_2) \leq C_2.
\end{align}
Consider the barrier function $\phi(x) := (|x|^{-p} - (3/4)^{-p})/p$ in the region $\{ |x| > 1/4 \}$ with large $p$ to be determined later and extend $\phi$ smoothly. For large $\tilde{C_2}$ and each vertex point $y_V \in B_{1/4} \cap B_{3/4}(x_{1})$, we define the test function 
\begin{align*}
    \psi(x) := P_{1}(x) + \tilde{C_2} \phi(x-y_V).
\end{align*}
Let $x_T \in \overline{B}_{3/4}(y_V)$ be a touching point such that
\begin{align} \label{Case2vis1}
    (u-\psi)(x_T)=\min_{\overline{B}_{3/4}(y_V)}(u-\psi).
\end{align}
Then we get $x_T \in B_{3/4}(y_V)$ since
\begin{align*}
    & (u-\psi)|_{\partial B_{3/4}(y_V)} \geq 0, \\
    & (u-\psi)(x_T) \leq (u-\psi)(x_1) = -\tilde{C_2}\phi(x_1) <0.
\end{align*}
\textbf{Claim}: For large enough $\tilde{C_2} =\tilde{C_2}(n,\lambda,\Lambda,\gamma)>1$ and small enough $\epsilon_2=\epsilon_2(n,\lambda,\Lambda,\gamma,\kappa)>0$, there exists a vertex point $y^{0}_V \in B_{1/4} \cap B_{3/4}(x_{1})$ such that a corresponding touching point $x^{0}_T$ is close to vertex point, $x^{0}_T \in B_{1/4}(y^0_V)$.

If we prove the \textbf{Claim}, we get $x^{0}_T \in B_{1/2}$ since
$$|x^{0}_T| \leq |x^{0}_T-y^{0}_V| + |y^{0}_V| < \frac{1}{4} + \frac{1}{4} = \frac{1}{2}.$$
Hence we show \eqref{Case2Step1} by $x_2=x^{0}_T$ and $C_2 =\tilde{C_2} \norm{\phi}_{L^{\infty}}$ since
$$u(x^{0}_T) < \phi(x^{0}_T) \leq P_{1}(x^{0}_T) + \tilde{C_2}\norm{\phi}_{L^{\infty}}.$$
Suppose by contradiction that for all $y_V \in B_{1/4} \cap B_{3/4}(x_{1})$, the corresponding touching points are always in $x_T \in B_{3/4}(y_V) \setminus B_{1/4}(y_V)$. Then we have
\begin{align} \label{Case2Propx_T}
    \frac{1}{4} <|x_T-y_V| < \frac{3}{4}, \quad |x_T| < 1.
\end{align}
We define $T$ by the set of corresponding touching points $x_T$, then $T \subset B_1$.

The proof of \textbf{Claim} is divided into four minor steps.

\textbf{Step 1-1.}
From \eqref{Case2vis1}, we see that
$$ \psi + \min_{\overline{B}_{3/4}(y_V)}(u-\psi) \overset{x_T}{\leq} u \quad \text{ in } \overline{B}_{3/4}(y_V). $$
By the definition of viscosity solution, we obtain
\begin{align*}
    |D\psi(x_T)|^{\gamma} \mathcal{P}_{\lambda,\Lambda}^{-}(D^{2}\psi(x_T)) \leq f(x_T).
\end{align*}
A direct computation with \eqref{Case2Propx_T} shows
\begin{align} \label{Case2Dpsi}
    \nonumber |D\psi(x_T)| &= \left|-|x_T - y_1|^{\alpha} \frac{x_T - y_1}{|x_T - y_1|} -\frac{\tilde{C_2}}{|x_T-y_V|^{p+1}} \frac{x_T-y_V}{|x_T-y_V|}\right| \\
    &\geq \tilde{C_2}   \left(\frac{4}{3}\right)^{p+1} - (C_0 + 1)^{\alpha} \ \geq 1,
\end{align}
where $\tilde{C_2}(p,C_0)$ is large enough.
Note that $1 < |x_T-y_1|$ since $|y_1| >2, x_T \in B_1$.
Also,
\begin{align*}
\mathcal{P}^{-}(D^{2}\psi(x_T)) &\geq \mathcal{P}^{-}(D^{2}P_{1}(x_T)) + \tilde{C_2}{P}^{-}(D^{2}\phi(x_T-y_V)) \\
 &= - \frac{\Lambda(n-1+\alpha)}{|x_T - y_1|^{1-\alpha}} +\frac{\tilde{C_2}}{|x_T-y_V|^{p+2}}((p+1)\lambda - (n-1)\Lambda) \\
 &\geq -\Lambda(n-1+\alpha) + \tilde{C_2} \left(\frac{4}{3}\right)^{p+2}((p+1)\lambda - (n-1)\Lambda) \ \geq 1
\end{align*}
for large enough $p=p(n,\lambda,\Lambda,\gamma)$.
Hence, we have
\begin{align*}
    1 \leq |D\psi(x_T)|^{\gamma} \mathcal{P}_{\lambda,\Lambda}^{-}(D^{2}\psi(x_T)) \leq f(x_T), \quad \forall x_T \in T.
\end{align*}
Finally, note that $Du$ is Lipschitz on touching point $T$ by recalling that $D^2\psi$ is bounded and using the method in \cite{Li18}. Therefore, by Radamacher theorem, $D^2u$ exists almost everywhere on $T$.\\

\textbf{Step 1-2.} For any $x_T \in T$, we have
\begin{align*}
    Du(x_T)= DP_1(x_T) + \tilde{C_2} D\phi(x_T-y_V), \\
    D^{2}u(x_T) \geq D^{2}P_1(x_T) + \tilde{C_2} D^{2}\phi(x_T-y_V).
\end{align*}
Let $v = Du-DP_1$. Then we obtain
\begin{align*}
     y_V &= x_T - (D\phi)^{-1}( \tilde{C_2}^{-1}Dv(x_T)).
\end{align*}
By differentiating with respect to $x_T$ and using the inverse function theorem, we get
\begin{align*}
    D_{x_T}y_V = \frac{\tilde{C_2}^{\frac{1}{p+1}}}{{|Dv|^{1+\frac{1}{p+1}}}} \left(I-\left(\frac{1}{p+1} + 1\right) \frac{Dv}{|Dv|} \otimes \frac{Dv}{|Dv|}\right)(D^{2}v(x_T)-\tilde{C_2} D^{2}\phi(x_T-y_V)).
\end{align*}
Consequently, we have
\begin{align*}
    \det D_{x_T}y_V = -\frac{1}{p+1}\left(\frac{\tilde{C_2}^{\frac{1}{p+1}}}{|Dv|^{1+\frac{1}{p+1}}}\right)^{n} \det(D^{2}v(x_T)- \tilde{C_2} D^{2}\phi(x_T-y_V)).
\end{align*}
\textbf{Step 1-3.} We want to find a bound for $\det(D^{2}v(x_T)-\tilde{C_2} D^{2}\phi(x_T-y_V))$ by using \eqref{subPDE}. Note that
\begin{align*}
     |Du|^{\gamma} \mathcal{P}^{-}(D^{2}u - D^{2}\psi) +|Du|^{\gamma} \mathcal{P}^{-}(D^{2}\psi) \leq |Du|^{\gamma} \mathcal{P}^{-}(D^{2}u) \leq f.
\end{align*}
From \textbf{Step 1-1}, we have $Du(x_T) = D\psi(x_T)$, $|D\psi(x_T)|^{\gamma} \mathcal{P}^{-}(D^{2}\psi(x_T)) \geq 0$, and
\begin{align*}
    |Du(x_T)|^{\gamma} \mathcal{P}^{-}(D^{2}u - D^{2}\psi)(x_T) \leq f(x_T).
\end{align*}
Since $D^{2}u(x_T) - D^{2}\psi(x_T) = D^{2}v(x_T)-\tilde{C_2} D^{2}\phi(x_T-y_V) \geq 0$, we get
\begin{align*}
    & \lambda \tr(D^{2}u(x_T) - D^{2}\psi(x_T)) = \mathcal{P}^{-}(D^{2}u(x_T) - D^{2}\psi(x_T)) \leq \frac{f(x_T)}{|Du(x_T)|^{\gamma}},\\
    & \det(D^{2}v(x_T)-\tilde{C_2} D^{2}\phi(x_T-y_V)) \leq \left(\frac{\tr(D^{2}u(x_T) - D^{2}\psi(x_T))}{n}\right)^{n} \leq C\left(\frac{f(x_T)}{|Du(x_T)|^{\gamma}}\right)^{n}.
\end{align*}
From \eqref{Case2Dpsi}, we also have $ |Du(x_T)| = |D\psi(x_T)| >1$, $|Dv(x_T)| = \tilde{C_2}|D\phi(x_T-y_V)| > C \tilde{C_2}$. Thus, we deduce that
\begin{align*}
    |\det D_{x_T}y_V| &\leq C \left|\frac{f(x_T)}{|Dv(x_T)|^{1+\frac{1}{p+1}}|Du(x_T)|^{\gamma}}\right|^{n}\\
    &\leq C|f(x_T)|^{n}.
\end{align*}

\textbf{Step 1-4.} Consider $m : T \rightarrow B_{1/4} \cap B_{3/4}(x_{1})$ by $m(x_T) =y_V$. Then by area formula,
\begin{align*}
    |B_{1/4} \cap B_{3/4}(x_{1})| \leq \int_{T}|\det D_{x_T}y_V| dx_T \leq C \int_{T} |f(x_T)|^{n} dx_T.
\end{align*}
Since $B_{1/4} \cap B_{3/4}(x_{1})$ contains a ball of radius $\kappa/2$ and $T \subset B_1 \subset B_2(x_1)$, we can conclude that
\begin{align*}
    \left( \frac{\kappa}{2} \right)^{n}|B_1| &\leq C \int_{B_2(x_1)} |f(x_T)|^{n} dx_T\\
    &\leq C\mathcal{M}(|f|^n)(x_1)|B_2(x_1)| \\
    &\leq C\epsilon_2 2^n |B_1|,
\end{align*}
which makes contradiction for small enough $\epsilon_{2} < 4^{-n}C^{-1}\kappa^n$. This proves the \textbf{Claim} and finishes \textbf{Step 1}.\\

\textbf{Step 2.} As in \textbf{Case 1}, we want to find universal $M_2=M_2(n,\lambda,\Lambda,\gamma)>1$ and $V_2 \subset \Omega$ with positive measure such that $T^{-}_{M_2}(V_2) \subset B_1$.

For each $\tilde{x} \in T^{-}_{M_2}(V_2)$, there exists $\tilde{y} \in V_2$ such that
\begin{align*}
    P_{M_2}(x) = -\frac{M_2}{1+\alpha}|x-\tilde{y}|^{1+\alpha} + C \overset{\tilde{x}}{\leq} u \quad \text{ in } \Omega.
\end{align*}
Consider the cone-like $Q$ defined by
\begin{align*}
    Q(x)& = P_{M_2}(x)- P_1(x)\\
    &= -\frac{M_2}{1+\alpha}|x-\tilde{y}|^{1+\alpha} + \frac{1}{1+\alpha}|x-y_1|^{1+\alpha} +C'.
\end{align*}
Then, as in \textbf{Case 1}, we have
\begin{align*}
    Q(x_2) \leq Q(\tilde{x}) + C_{2}.
\end{align*}

We choose $V_2 \subset \Omega$ as
$$V_2 = \overline{B}_{\frac{1}{8}\frac{M_2^{1/\alpha}-1}{M_2^{1/\alpha}}} \left(\frac{1}{M_2^{1/\alpha}}y_1 + \frac{M_2^{1/\alpha}-1}{M_2^{1/\alpha}}x_2 \right).$$
Then vertex of $Q$ is in $y_0 \in B_{1/8}(x_2)$ if $\tilde{y} \in V_2.$

\textbf{Claim}: For large $M_2$ only depend on $n, \lambda, \Lambda$, we have $\tilde{x} \in B_{1/4}(y_0)$.

We assume $y_0=0$ and $Q(0) = 0$ by translation and adding some constant. Then we only need to prove $|\tilde{x}| \leq \frac{1}{4}$.
Note that we have $|x_2| < \frac{1}{8}$, $\tilde{y} = \frac{1}{M^{1/{\alpha}}} y_1$, and
\begin{align*}
    Q(x) = -\frac{M_2}{1+\alpha} \left|x-\frac{1}{M_2^{1/{\alpha}}} y_1 \right|^{1+\alpha} +\frac{1}{1+\alpha}|x-y_1|^{1+\alpha} + C'.
\end{align*}
Define $a_{Q}$ by,
\begin{align*}
    a_{Q} = \sup_{x \in \mathbb{R}^n}\{ |x| : Q(x) \geq \min_{B_{1/8}} Q - C_2\}.
\end{align*}
Then, it is enough to show $a_{Q} \leq 1/4$.
We consider
\begin{align*}
    \tilde{Q}(x) = \frac{Q(x)}{M_2} = -\frac{1}{1+\alpha}\left|x-\frac{1}{M_2^{1/{\alpha}}} y_1 \right|^{1+\alpha} +\frac{1}{M_2(1+\alpha)}|x-y_1|^{1+\alpha} + \frac{C'}{M_2}.
\end{align*}
Since $|y_1| \leq C_0$, $\tilde{Q}(0)=0$ and $D\tilde{Q}(0) = 0$, if we choose large $M_2 = M_2(C_0,\gamma)$, then $\tilde{Q}$ can be approximated by the $C^{1,\alpha}$ cone in $B_1$,
\begin{align*}
    \tilde{Q}(x) \, \longrightarrow \, -\frac{1}{1+\alpha}|x|^{1+\alpha} \quad\text{uniformly in}\ B_{1}, \text{ as } M_2 \rightarrow \infty.
\end{align*}
Therefore, we have
\begin{align*}
    \min_{B_{\frac{1}{8}}} \tilde{Q} &\leq -\frac{1}{1+\alpha} \left( \frac{1}{8} \right)^{1+\alpha}, \\
    a_{\tilde{Q}} &= \sup\left\{ |x| : -\frac{1}{1+\alpha}|x|^{1+\alpha} \geq -\frac{1}{1+\alpha}\left(\frac{1}{8}\right)^{1+\alpha} - \frac{C_2}{M_2}\right\}\\
    &\leq \sup\left\{ |x| : -\frac{1}{1+\alpha}|x|^{1+\alpha} \geq -\frac{1}{1+\alpha}\left(\frac{1}{4}\right)^{1+\alpha}\right\} =\frac{1}{4}
\end{align*}
for $M_2 = M_2(C_0,C_2,\gamma)$ large enough so that $-\frac{1}{1+\alpha}\left(\frac{1}{8}\right)^{1+\alpha} - \frac{C_2}{M_2} \geq -\frac{1}{1+\alpha}\left(\frac{1}{4}\right)^{1+\alpha}$.
Since $\tilde{Q}$ satisfies \eqref{cone}, we do not need to consider $\tilde{Q}$ outside of $B_1$.
Consequently, $a_Q \leq 1/4$, which prove $T^{-}_{M_2}(V_2) \subset B_1$.

\textbf{Case 3.} $|y_{1}| < 2$

\textbf{Step 1.} We prove that there exist a concave $C^{1,\alpha}$ cone $P_{\tilde{M}_3}$ with opening $\tilde{M}_3 = \tilde{M}_3(n,\lambda,\Lambda,\gamma)>0$ and vertex $y_2 \in B_1$ touching $u$ from below at $x_2 \in B_{1/2}$,
\begin{align*}
    P_{\tilde{M}_3}(x) = -\frac{\tilde{M}_3}{1+\alpha}|x-y_2|^{1+\alpha} +C \overset{x_2}{\leq} u(x)  \quad \text{ in } \Omega.
\end{align*}
Consider the barrier function $\phi(x):=(|x|^{-p} - (3/4)^{-p})/p$ for $|x| > 1/4$ with large $p$ and extend $\phi$ so that it is smooth, radially symmetric and decreasing.
For large $C_3$ and each vertex point $y_V \in B_{1/4} \cap B_{3/4}(x_{1})$, we define the test function 
\begin{align*}
    \psi(x) = C_3 \phi(x-y_V) + P_1(x_1).
\end{align*}
Let $x_T \in \overline{B}_{5/4}(y_V)$ be a touching point such that
\begin{align*}
    (u-\psi)(x_T)=\min_{\overline{B}_{5/4}(y_V)}(u-\psi).
\end{align*}
Then for large $C_3 = C_3(p,\gamma) > 0$,
\begin{align} \label{Case3min}
    \nonumber (u-\psi)|_{\partial B_{5/4}(y_V)} &\geq (P_1-\psi)|_{\partial B_{5/4}(y_V)} \\
    \nonumber &\geq \min_{\partial B_{5/4}(y_V)} (P_1 -P_1(x_1)) +\frac{C_3}{p} \left( \left(\frac{4}{3}\right)^p -\left(\frac{4}{5}\right)^p \right) \\
    \nonumber &\geq -\osc_{B_{7/2}(y_1)}P_1 +\frac{C_3}{p}\left( \left(\frac{4}{3}\right)^p -\left(\frac{4}{5}\right)^p \right) \\
    &\geq -\frac{1}{1+\alpha}\left(\frac{7}{2}\right)^{1+\alpha}  +\frac{C_3}{p}\left( \left(\frac{4}{3}\right)^p -\left(\frac{4}{5}\right)^p \right) >0, \\
    (u-\psi)(x_T) &\leq (u-\psi)(x_1) = -\nonumber C_3\phi(x_1-y_V) <0,
\end{align}
since $B_{5/4}(y_V) \subset B_{7/2}(y_1)$.
Therefore, we get $x_T \in B_{5/4}(y_V)$.

\textbf{Claim}: For large enough $C_3=C_3(n,\lambda,\Lambda,\gamma)$ and small enough $\epsilon_2=\epsilon_2(n,\lambda,\Lambda,\gamma,\kappa)$, there exists a vertex point $y^{0}_V \in B_{1/4} \cap B_{3/4}(x_{1})$ such that the corresponding touching point is in $x^{0}_T \in B_{1/4}(y^0_V)$.

Suppose by contradiction that for all $y_V \in B_{1/4} \cap B_{3/4}(x_{1})$, the touching point is always in $x_T \in B_{5/4}(y_V) \setminus B_{1/4}(y_V)$. Then,
\begin{align*}
    \frac{1}{4} <|x_T-y_V| < \frac{5}{4}, \quad |x_T| < \frac{3}{2}.
\end{align*}
We define $T$ by the set of corresponding touching points $x_T$, then $T \subset B_2$.

The proof of \textbf{Claim} is divided into four minor steps.

\textbf{Step 1-1.}
From \eqref{Case1vis1}, we see that
$$ \psi + \min_{\overline{B}_{5/4}(y_V)}(u-\psi) \overset{x_T}{\leq} u \quad \text{ in } \overline{B}_{5/4}(y_V). $$
By the definition of viscosity solution, we obtain
\begin{align*}
    |D\psi(x_T)|^{\gamma} \mathcal{P}_{\lambda,\Lambda}^{-}(D^{2}\psi(x_T)) \leq f(x_T).
\end{align*}
A direct computation gives
\begin{align} \label{Case3Dpsi}
    |D\psi(x_T)| &= \left|\frac{C_3}{|x_T-y_V|^{p+1}} \frac{x_T-y_V}{|x_T-y_V|}\right| \geq C_3  \left(\frac{4}{5}\right)^{p+1} \ \geq 1,
\end{align}
where $C_3 = C_3(p)$ is large enough.
Also, we have
\begin{align*}
\mathcal{P}^{-}(D^{2}\psi(x_T)) &\geq C_3{P}^{-}(D^{2}\phi(x_T-y_V)) \\
 &= \frac{C_3}{|x_T-y_V|^{p+2}}((p+1)\lambda - (n-1)\Lambda) \\
 &\geq C_3 \left(\frac{4}{5}\right)^{p+2}((p+1)\lambda - (n-1)\Lambda) \ \geq 1
\end{align*}
for large enough $p=p(n,\lambda,\Lambda)$.
Hence, we have
\begin{align*}
     1 \leq |D\psi(x_T)|^{\gamma} \mathcal{P}_{\lambda,\Lambda}^{-}(D^{2}\psi(x_T)) \leq f(x_T), \quad \forall x_T \in T.
\end{align*}

\textbf{Step 1-2.} For any $x_T \in T$, we have
\begin{align*}
    Du(x_T)= C_3 D\phi(x_T-y_V), \\
    D^{2}u(x_T) \geq C_3 D^{2}\phi(x_T-y_V).
\end{align*}
Thus, we obtain
\begin{align*}
    y_V = x_T - (D\phi)^{-1}( C_3^{-1}Du(x_T)).
\end{align*}
By differentiating with respect to $x_T$ and using the inverse function theorem, we have
\begin{align*}
    D_{x_T}y_V = \frac{C_3^{\frac{1}{p+1}}}{{|Du|^{1+\frac{1}{p+1}}}}\left(I-\left(\frac{1}{p+1} + 1\right) \frac{Du}{|Du|} \otimes \frac{Du}{|Du|}\right)(D^{2}u(x_T)-C_3 D^{2}\phi(x_T-y_V)).
\end{align*}
Consequently, we have
\begin{align*}
    \det D_{x_T}y_V = -\frac{1}{p+1}\left(\frac{C_3^{\frac{1}{p+1}}}{|Dv|^{1+\frac{1}{p+1}}}\right)^{n} \det(D^{2}u(x_T)- C_3 D^{2}\phi(x_T-y_V)).
\end{align*}

\textbf{Step 1-3.} We want to find a bound for $\det(D^{2}u(x_T)-C_3 D^{2}\phi(x_T-y_V))$ by using \eqref{subPDE}. Note that
\begin{align*}
    |Du|^{\gamma} \mathcal{P}^{-}(D^{2}u - D^{2}\psi) +|Du|^{\gamma} \mathcal{P}^{-}(D^{2}\psi) \leq |Du|^{\gamma} \mathcal{P}^{-}(D^{2}u) \leq f .
\end{align*}
From \textbf{Step 1-1}, we have $Du(x_T) = D\psi(x_T)$, $|D\psi(x_T)|^{\gamma} \mathcal{P}^{-}(D^{2}\psi(x_T)) \geq 1$. Thus,
\begin{align*}
    |Du(x_T)|^{\gamma} \mathcal{P}^{-}(D^{2}u - D^{2}\psi)(x_T) \leq f(x_T).
\end{align*}
Since $D^{2}u(x_T) - D^{2}\psi(x_T) = D^{2}u(x_T)-C_3 D^{2}\phi(x_T-y_V) \geq 0$, we get
\begin{align*}
    & \lambda \tr(D^{2}u(x_T) - D^{2}\psi(x_T)) = \mathcal{P}^{-}(D^{2}u(x_T) - D^{2}\psi(x_T)) \leq \frac{f(x_T)}{|Du(x_T)|^{\gamma}},\\
    & \det(D^{2}u(x_T)-C_3 D^{2}\phi(x_T-y_V)) \leq\left(\frac{\tr(D^{2}u(x_T) - D^{2}\psi(x_T))}{n}\right)^{n} \leq\left(\frac{f(x_T)}{\lambda|Du(x_T)|^{\gamma}}\right)^{n}.
\end{align*}
From \eqref{Case3Dpsi}, we also have $ |Du(x_T)| = |D\psi(x_T)| >1$. Thus, we deduce that
\begin{align*}
    |\det D_{x_T}y_V| &\leq C \left|\frac{f(x_T)}{|Du(x_T)|^{1+\frac{1}{p+1}}|Du(x_T)|^{\gamma}}\right|^{n}\\
    &\leq C|f(x_T)|^{n}.
\end{align*}

\textbf{Step 1-4.} Consider $m : T \rightarrow B_{1/4} \cap B_{3/4}(x_{1})$ defined by $m(x_T) =y_V$. Then by area formula,
\begin{align*}
    |B_{1/4} \cap B_{3/4}(x_{1})| \leq \int_{T}|\det D_{x_T}y_V| dx_T \leq C \int_{T} |f(x_T)|^{n} dx_T.
\end{align*}
Since $B_{1/4} \cap B_{3/4}(x_{1})$ contains a ball of radius $\kappa/2$ and $T \subset B_2 \subset B_3(x_1)$, we can conclude that
\begin{align*}
    \left( \frac{\kappa}{2} \right)^{n}|B_1| &\leq C \int_{B_3(x_1)} |f(x_T)|^{n} dx_T\\
    &\leq C\mathcal{M}(|f|^n)(x_1)|B_3(x_1)| \\
    &\leq C\epsilon_2 3^n |B_1|,
\end{align*}
which makes contradiction for small enough $\epsilon_{2} < 6^{-n}C^{-1}\kappa^n$. This proves the \textbf{Claim}.

Therefore, we know that there exist $x_2 = x^{0}_{T} \in B_{1/4}(y^{0}_V) \subset B_{1/2}$ such that $u \overset{x_2}{\geq} \psi + C'$ in $\overline{B}_{5/4}(y^{0}_V)$, where $C' =\min_{\overline{B}_{5/4}(y_V)}(u-\psi) <0$ by \eqref{Case3min}.
Since $\psi$ is smooth, there exist a $C^{1,\alpha}$ cone $P_{\tilde{M}_3}$ with opening $\tilde{M}_3>1$ depending only on smoothness of $\psi$ and a vertex $y_2$ such that
\begin{align*}
    u(x) \overset{x_2}{\geq} \psi(x) + C' \overset{x_2}{\geq} P_{\tilde{M}_3}(x) := -\frac{\tilde{M}_3}{1+\alpha}|x-y_2|^{1+\alpha} + C \quad \text{ in } \overline{B}_{5/4}(y^{0}_V). 
\end{align*}
We claim that the vertex of $P_{\tilde{M}_3}$ is in $y_2 \in B_{1/4}(y^{0}_V) \subset B_{1/2} \subset \Omega$.
Since $\psi$ is radially symmetric at center $y^{0}_V$, $P_{\tilde{M}_3}$ is radially symmetric at center $y_2$, and $D\psi(x_2) = DP_{\tilde{M}_3}(x_2)$, we know that $y_2$ is on the line $\overleftrightarrow{y^{0}_V x_2}$.
We claim that $y_2$ is on the line segment $\overline{y^{0}_V x_2} \subset B_{1/4}(y^{0}_V)$.

If $y_2$ is on the ray starting at $x_2$ opposite to $y^{0}_V$, this makes a contradiction since the direction of $D\psi(x_2)$ is $y^{0}_V - x_2$ but $DP_{\tilde{M}_3}(x_2)$ is $y_2 -x_2$, which is opposite direction.
If $y_2$ is on the ray starting at $y^{0}_V$ opposite to $x_2$, Observe that $|y_2 -x_2| > |y^{0}_V -x_2|$, thus for a point $x_3$ such that $\frac{1}{2}(x_2 + x_3) = y_2$, we have $|y^{0}_V-x_3| > |y^{0}_V-x_2|$. However, $P_{\tilde{M}_3}$ and $\psi$ is radically symmetric and decreasing, $P_{\tilde{M}_3}(x_3) = P_{\tilde{M}_3}(x_2) = \psi(x_2) + C' > \psi(x_3) + C'$ which contradicts with $\psi + C' \geq P_{\tilde{M}_3}$.

Therefore, we find a $C^{1,\alpha}$ cone $P_{\tilde{M}_3}$ which almost satisfies the condition of \textbf{Step 1}, but we know $P_{\tilde{M}_3} \overset{x_2}{\leq} u$ only in $\overline{B}_{5/4}(y^{0}_V)$, not in $\Omega$.
Note that from \eqref{Case3min},
\begin{align*}
    &P_{\tilde{M}_3} \leq \psi+C' < \psi < P_1 \quad \text{ on } \, \partial B_{5/4}(y_V), \\
    &P_{\tilde{M}_3}(x_2) = u(x_2) \geq P_1(x_2).
\end{align*}
Define $Q(x) := P_{\tilde{M}_3}(x) - P_1(x)$. Then $Q(x) < 0$ on $\partial B_{5/4}(y_V)$ and $Q(x_2) \geq 0$.
Thus, there exists a local maximum point $y \in B_{5/4}(y_V)$ such that $DQ(y) = 0$.
Since $Q$ satisfies \eqref{cone2}, we conclude from \eqref{cone} that $y$ is the vertex of $Q$ and $\max_{\mathbb{R}^{n} \setminus B_{5/4}(y_V)} Q = \max_{\partial B_{5/4}(y_V)} Q <0$, which implies $P_{\tilde{M}_3} \leq P_1 \leq u$ in $\Omega$ and \textbf{Step 1} is finished.

\textbf{Step 2.} As in \textbf{Case 2}, we want to find a universal $M_3 = M_3(n,\lambda,\Lambda,\gamma)>\tilde{M}_3$ and a closed ball $V_3 \subset \Omega$ such that $T^{-}_{M_3}(V_3) \subset B_1$.
This can be done as in \textbf{Case 2} with $P_1$ replaced by $P_{\tilde{M}_3}$. 
More precisely, for each $\tilde{x} \in T^{-}_{M_3}(V_3)$, there exists $\tilde{y} \in V_3$ such that
\begin{align*}
    P_{M_3}(x) = -\frac{M_3}{1+\alpha}|x-\tilde{y}|^{1+\alpha} + C \overset{\tilde{x}}{\leq} u \quad \text{ in } \Omega.
\end{align*}
Consider a cone-like $Q$ defined by
\begin{align*}
    Q(x)& = P_{M_3}(x)- P_{\tilde{M}_3}(x)\\
    &= -\frac{M_3}{1+\alpha}|x-\tilde{y}|^{1+\alpha} +\frac{\tilde{M}_3}{1+\alpha}|x-y_1|^{1+\alpha} +C',
\end{align*}
where $C'$ is a constant. Then we have $Q(x_2) \leq Q(\tilde{x})$.
We set $V_3 \subset \Omega$ as
$$V_3 = \overline{B}_{\frac{1}{8}\frac{M_3^{1/\alpha}-\tilde{M}_3^{1/\alpha}}{M_3^{1/\alpha}}} \left(\frac{\tilde{M}_3^{1/\alpha}}{M_3^{1/\alpha}}y_1 + \frac{M_3^{1/\alpha}-\tilde{M}_3^{1/\alpha}}{M_3^{1/\alpha}}x_2 \right).$$
Then the vertex of $Q$ is $y_0 \in B_{1/8}(x_2)$ if $\tilde{y} \in V_3.$
With a proper translation, we assume that $DQ(0)=0$.
Then one can prove that $a_{Q} = \sup\{ |x| : Q(x) \geq \min_{B_{1/8}} Q\} \leq 1/4$ by choosing $M_3 \gg \tilde{M}_3$ so large that $Q$ can be approximated by a $C^{1,\alpha}$ cone as in \textbf{Case 2}.

\textbf{Step 3.} So far, we have proved that for all three cases $i \in \{1,2,3\}$, there exist a universal constant $M_i>1$ and a vertex subset $V_i \subset \Omega$ such that $T^{-}_{M_i}(V_i) \subset B_1$.
We next assert that
\begin{equation} \label{Step3_Claim}
    |V_i| \leq C|T^{-}_{M_i}(V_i)|
\end{equation}
for some constant $C=C(n,\lambda,\Lambda)>0$.

For simplicity, we write $M,V$ instead of $M_i, V_i$. For each $x \in T^{-}_{M}(V)$, there exists $y\in V$ such that
\begin{align*}
    P(z-y) = -\frac{M}{1+\alpha}|z-y|^{1+\alpha} + C \overset{x}{\leq} u(z) \quad \text{ in } \Omega.
\end{align*}
Therefore, we know that
\begin{align*}
    Du(x) = DP(x-y), \quad D^2u(x) \geq D^2P(x-y).
\end{align*}
Since $(DP)^{-1}(z) = -\frac{1}{M^{1+\gamma}} |z|^\gamma z$, the mapping $m : T^{-}_{M}(V) \ni x \mapsto y \in V$ is precisely given by
\begin{equation} \label{Step3xy}
    y = x - (DP)^{-1}(Du(x)) = x + \frac{1}{M^{1+\gamma}}|Du|^\gamma Du(x).
\end{equation}
Now, we want to prove that $Du$ is Lipschitz in $T^{-}_{M}(V)$ so that the mapping $m : x\mapsto y$ is also Lipschitz.
Unlike \textbf{Step 1}, we are going to consider $T^{-,\epsilon}_{M}(V)$ for every small $\epsilon>0$, instead of $T^{-}_{M}(V)$, where
\begin{equation*}
    T^{-,\epsilon}_{M}(V) := \{ x \in T^{-}_{M}(V) : |x-y| > \epsilon, \text{ where } y \text{ is the corresponding vertex point of }x \}.
\end{equation*}
We also define
\begin{align*}
    T^{-,0+}_{M}(V) &:= \{ x \in T^{-}_{M}(V) : |x-y| > 0, \text{ where } y \text{ is the corresponding vertex point of }x \},\\
    T^{-,0}_{M}(V) &:= \{ x \in T^{-}_{M}(V) : x=y, \text{ where } y \text{ is the corresponding vertex point of }x \},
\end{align*}
and $V^{\epsilon} := m(T^{-,\epsilon}_{M}(V))$, $V^{0+} := m(T^{-,0+}_{M}(V))$, $V^{0} := m(T^{-,0}_{M}(V))$.

It is clear that for any $\epsilon_1 > \epsilon_2 > 0$, 
\begin{align*}
    T^{-,\epsilon_1}_{M}(V) \subset T^{-,\epsilon_2}_{M}(V) \qquad \text{and} \qquad V^{\epsilon_1} \subset V^{\epsilon_2}.
\end{align*}
Also,
\begin{align*}
    T^{-,0+}_{M}(V)&=\bigcup_{\epsilon>0}T^{-,\epsilon}_{M}(V), \quad T^{-}_{M}(V) = T^{-,0}_{M}(V) \sqcup T^{-,0+}_{M}(V),\\
   V^{0+}&= \bigcup_{\epsilon>0}V^{\epsilon}, \quad V = V^{0} \cup V^{0+}.
\end{align*}
For each $x \in T^{-,\epsilon}_{M}(V)$, there exists $y \in V^{\epsilon}$ such that $|x-y|>\epsilon$ and
\begin{align*}
    P(z-y) = -\frac{M}{1+\alpha}|z-y|^{1+\alpha} + C \overset{x}{\leq} u(z) \quad \text{ in } \Omega.
\end{align*}
We consider a $C^2$-function $P^{\epsilon}$, constructed by mollifying $P$ near the vertex $0$, with
\begin{align*}
    & P^{\epsilon}(z) = P(z) \quad \text{ if } z \notin B_\epsilon,\\
    & P^{\epsilon}(z) \leq P(z) \quad \text{and} \quad |D^2P^\epsilon(z)| \leq CM\epsilon^{-1} \quad \text{ if } z \in B_\epsilon.
\end{align*}
Since $x\notin B_\epsilon(y)$, we obtain
\begin{align*}
    P^{\epsilon}(z-y) \overset{x}{\leq} u(z) \quad \text{ in } \Omega.
\end{align*}
Thus, for any $z \in \Omega$ and $x \in T^{-,\epsilon}_{M}(V)$, we get
\begin{align*}
    u(z) \geq P^{\epsilon}(z-y) &\geq P^{\epsilon}(x-y) + \langle DP^{\epsilon}(x-y),z-x \rangle - CM\epsilon^{-1}|z-x|^2\\
    &=u(x) + \langle Du(x),z-x \rangle - CM\epsilon^{-1}|z-x|^2
\end{align*}
since $|D^2P(z)| \leq CM\epsilon^{-1}$ and $Du(x)=DP^{\epsilon}(x-y)$.
For any $x_1, x_2 \in T^{-,\epsilon}_{M}(V)$, we set $r:=2|x_1-x_2|$.
For any $z \in B_r(x_1)$, we have
\begin{align*}
    u(z) & \geq u(x_1) + \langle Du(x_1),z-x_1 \rangle - CM\epsilon^{-1}|z-x_1|^2 \\
    & \geq u(x_2) + \langle Du(x_2),x_1-x_2 \rangle - CM\epsilon^{-1}|x_1-x_2|^2 + \langle Du(x_1),z-x_1 \rangle - CM\epsilon^{-1}|z-x_1|^2.
\end{align*}
By the semi-concavity of $u$, there is a universal constant $\tilde{C}>0$ such that 
\begin{align*}
    u(z) \leq u(x_2) + \langle Du(x_2),z-x_2 \rangle + \tilde{C}|z-x_2|^2.
\end{align*}
Therefore, for any $z \in B_r(x_1)$, we have
\begin{align*}
    \langle Du(x_1)-Du(x_2),z-x_1 \rangle \leq C(\tilde{C} + M\epsilon^{-1})r^2.
\end{align*}
By taking $z$ such that $z-x$ is parallel to $Du(x_1)-Du(x_2)$ and $|z-x| = r/2$, we conclude that
\begin{align*}
    |Du(x_1)-Du(x_2)|\leq C(\tilde{C} + M\epsilon^{-1})|x_1-x_2|,
\end{align*}
which implies that $Du$ is Lipschitz on $T^{-,\epsilon}_{M}(V)$.
By Rademacher theorem, $u$ is second order differentiable almost everywhere on $T^{-,\epsilon}_{M}(V)$.
Note that for $x \in T^{-,\epsilon}_{M}(V)$, we get $M^{-1-\gamma}|Du(x)|^{1+\gamma} = |y-x|>0$, which implies $|Du(x)|>0$. 

We differentiate \eqref{Step3xy} with respect to $x$ to find
\begin{align*}
    D_x y &= I + \frac{1}{M^{1+\gamma}} D(|Du|^\gamma Du)\\
    &= I + \frac{1}{M^{1+\gamma}}|Du|^\gamma\left(I+\gamma \frac{Du}{|Du|} \otimes \frac{Du}{|Du|}\right)D^2 u(x).
\end{align*}
Note that $\frac{1}{M^{1+\gamma}}|DP|^\gamma DP(z) =-z, \ \forall z \in \mathbb{R}^n$. Therefore,
\begin{align*}
     \frac{1}{M^{1+\gamma}}D(|DP|^\gamma DP)(x-y)= \frac{1}{M^{1+\gamma}}|DP|^\gamma\left(I+\gamma \frac{DP}{|DP|} \otimes \frac{DP}{|DP|}\right)D^2 P(x-y) = -I.
\end{align*}
Since $Du(x) = DP(x-y)$, we conclude that
\begin{align*}
    D_x y = \frac{1}{M^{1+\gamma}}|Du|^\gamma\left(I+\gamma \frac{Du}{|Du|} \otimes \frac{Du}{|Du|}\right)(D^2 u(x) - D^2 P(x-y)).
\end{align*}
Since $\det\left(I+\gamma \frac{Du}{|Du|} \otimes \frac{Du}{|Du|}\right) = 1+\gamma$, we obtain
\begin{align*}
    \det(D_x y) = (1+\gamma)\det\left(\frac{1}{M^{1+\gamma}}|Du|^\gamma(D^2 u(x) - D^2 P(x-y))\right) =: (1+\gamma)\det(A).
\end{align*}
Observe that
\begin{align*}
    D^2 u(x) &\geq D^2 P(x-y) = -\frac{M}{|x-y|^{1-\alpha}}\left(I+(1-\alpha)\frac{x-y}{|x-y|} \otimes \frac{x-y}{|x-y|}\right).
\end{align*}
Using \eqref{Step3xy}, we get
\begin{align*}
    D^2 P(x-y) = -\frac{M^{1+\gamma}}{|Du|^\gamma}\left(I+\frac{\gamma}{1+\gamma}\frac{Du}{|Du|} \otimes \frac{Du}{|Du|}\right).
\end{align*}
Therefore,
\begin{align*}
    A &= \frac{1}{M^{1+\gamma}}|Du|^\gamma(D^2 u(x) - D^2 P(x-y))\\
    &= \frac{1}{M^{1+\gamma}}|Du|^\gamma D^2 u + \left(I+\frac{\gamma}{1+\gamma}\frac{Du}{|Du|} \otimes \frac{Du}{|Du|}\right).
\end{align*}
Note that $A \geq 0$ since $D^2 u(x) \geq D^2 P(x-y)$. Thus,
\begin{align*}
    \lambda \tr A &= \mathcal{P}^{-}(A) \\
    &\leq \frac{1}{M^{1+\gamma}}|Du|^\gamma \mathcal{P}^{-}(D^{2}u(x)) + \mathcal{P}^{+}\left(I+\frac{\gamma}{1+\gamma}\frac{Du}{|Du|} \otimes \frac{Du}{|Du|}\right) \\
    &\leq \frac{1}{M^{1+\gamma}} f(x) + \Lambda \left(n+\frac{\gamma}{1+\gamma}\right),\\
    \det A &\leq \left(\frac{\tr A}{n} \right)^n \leq C\left(1 + \frac{|f(x)|^n}{M^{(1+\gamma)n}} \right).
\end{align*}
Hence, we have
\begin{align*}
     0 \leq \det(D_x y) \leq C\left(1 + \frac{|f(x)|^n}{M^{(1+\gamma)n}}\right).
\end{align*}
Consequently, by applying the area formula to the mapping $m : T^{-,\epsilon}_{M}(V) \rightarrow V^{\epsilon}$, we conclude that
\begin{align*}
     |V^{\epsilon}| \leq \int_{T^{-,\epsilon}_{M}(V)} \det(D_x y) \, dx \leq C|T^{-,\epsilon}_{M}(V)| + \frac {C}{M^{(1+\gamma)n}} \int_{T^{-,\epsilon}_{M}(V)} |f(x)|^n \, dx.
\end{align*}

Since we proved at \textbf{Step 2} that $T^{-,\epsilon}_{M}(V) \subset T^{-}_{M}(V) \subset B_1 \subset B_2(x_1)$, we obtain
\begin{align*}
     |V^{\epsilon}| &\leq C|T^{-,\epsilon}_{M}(V)| + \frac {C}{M^{(1+\gamma)n}} \int_{B_2(x_1)} |f(x)|^n \, dx \\
     &\leq C|T^{-,\epsilon}_{M}(V)| + \frac {C}{M^{(1+\gamma)n}} \mathcal{M}(|f|^n)(x_1) \\
     &\leq C|T^{-,\epsilon}_{M}(V)| + C\epsilon_2,
\end{align*}
where $C$ is independent of $\epsilon$. By the continuity of the Lebesgue measure, letting $\epsilon \rightarrow 0$ gives
\begin{align*}
      |V^{0+}|= \lim_{\epsilon \rightarrow 0}|V^{\epsilon}| \leq C\lim_{\epsilon \rightarrow 0}|T^{-,\epsilon}_{M}(V)| + C\epsilon_2  = C|T^{-,0+}_{M}(V)| + C\epsilon_2.
\end{align*}
Recalling that $V^{0}=T^{-,0}_{M}(V)$, we have
\begin{align*}
      |V| \leq |V^{0}|+|V^{0+}| \leq |T^{-,0}_{M}(V)| + C|T^{-,0+}_{M}(V)| + C\epsilon_2 \leq C|T^{-}_{M}(V)| + C\epsilon_2.
\end{align*}
Choosing $\epsilon_2>0$ so small that $C\epsilon_2 \leq \frac{1}{2}\min\{|V_1|,|V_2|,|V_3|\} \leq \frac{1}{2}|V|$, we get
\begin{align*}
    \frac{1}{2}|V| \leq C|T^{-}_{M}(V)|,
\end{align*}
which proves \eqref{Step3_Claim}.

Then we finally conclude that
\begin{align*}
    |B_1 \cap T^{-}_{M}| \geq |B_1 \cap T^{-}_{M_i}| \geq |T^{-}_{M_i}(V_i)| \geq \frac{1}{C}|V_i| \geq \frac{1}{C}\min\{|V_1|,|V_2|,|V_3|\}|B_1| = \Theta |B_1|,
\end{align*}
which proves the lemma by taking $M = \max\{M_1,M_2,M_3\}$, $\Theta = \frac{1}{C}\min\{|V_1|,|V_2|,|V_3|\}$ and $\theta = \frac{1}{2}\min(\theta_0,\Theta) <1$.
\end{proof}

\begin{lem} \label{key2}
There is a universal constant $\theta_0=\theta_0(n,\lambda, \Lambda)>0$ so that for any $\epsilon \in (0,1)$, there exists $\epsilon_{1} = \epsilon_{1}(n,\lambda, \Lambda, \epsilon)>0$ such that if $u \in C(B_2)$ satisfies 
\begin{align*}
|Du|^{\gamma} \mathcal{P}_{\lambda,\Lambda}^{-}(D^{2}u) \leq f \quad\text{in} \ B_{2}
\end{align*}
in the viscosity sense with $\norm{f}_{L^{n}(B_{2})} \leq \epsilon_{1}$ and $\osc_{B_2} u \leq 1/8$, then 
$$|B_{1} \cap T^{-}_{1} \cap \{x \in B_1 : \mathcal{M}(|f|^n)(x) \leq \epsilon\}| \geq \theta_0 |B_{1}|.$$
\end{lem}

\begin{proof}
For each $\tilde{x} \in T^{-}_{1}(\overline{B_{1/2}})$, there exists $\tilde{y} \in \overline{B_{1/2}}$ such that
\begin{align*}
    -\frac{1}{1+\alpha}|x-\tilde{y}|^{1+\alpha} + C(\tilde{y}) \overset{\tilde{x}}{\leq} u(x) \quad \text{in } B_2.
\end{align*}
Note that $C(\tilde{y}) \leq u(\tilde{y})$ and $-\frac{1}{1+\alpha}|\tilde{x}-\tilde{y}|^{1+\alpha} + C(\tilde{y}) =u(\tilde{x})$. Hence, we have
\begin{align*}
    \frac{1}{1+\alpha}|\tilde{x}-\tilde{y}|^{1+\alpha} \leq u(\tilde{y}) - u(\tilde{x}) \leq \osc_{B_2} u \leq \frac{1}{8}.
\end{align*}
Therefore, $|\tilde{x} - \tilde{y}| \leq 1/2$, thus $\tilde{x} \in B_1$ and 
\begin{align*}
    T^{-}_{1}(\overline{B_{1/2}}) \subset B_1.
\end{align*}
Recall that in the proof of Lemma \ref{Key}, the mapping $T^{-}_{1}(\overline{B_{1/2}}) \ni \tilde{x} \mapsto \tilde{y} \in \overline{B_{1/2}}$ is precisely given by
\begin{align*}
    \tilde{y} = \tilde{x} + |Du|^\gamma Du(\tilde{x}).
\end{align*}
As Step 3 in the proof of Lemma \ref{Key}, we can obtain that
\begin{align*}
    0\leq \det(D_{\tilde{x}}\tilde{y}) \leq C(1 + |f(x)|^n).
\end{align*}
Since $T^{-}_{1}(\overline{B_{1/2}}) \subset B_1$, it follows from the area formula that
\begin{align*}
    2^{-n}|B_1|=|B_{1/2}| \leq C_0|T_{1}(\overline{B_{1/2}})| + C_0\int_{B_1}|f|^n \, dx
\end{align*}
for some constant $C_0>1$.
Let $\epsilon_1^n \leq 2^{-(n+1)}|B_1|C_0^{-1}$. Then we have
\begin{align*}
     |B_1 \cap T^{-}_1| \geq |T^{-}_{1}(\overline{B_{1/2}})| \geq 2^{-(n+1)}C_0^{-1}|B_1| =: 2\theta_0|B_1|.
\end{align*}
By the weak type $(1,1)$ property, we obtain
\begin{align*}
    |\{x \in B_1 : \mathcal{M}(|f|^n)(x) > \epsilon\}| \leq C_1 \epsilon^{-1} \norm{f}^n_{L^n(B_1)} \leq C_1 \epsilon^{-1} \epsilon_1^n \leq \theta_0|B_1|
\end{align*}
by choosing $\epsilon_1>0$ so small that $\epsilon_1^n \leq \min\{2^{-(n+1)}|B_1|C_0^{-1}, \theta_0|B_1| C_1^{-1} \epsilon\}$.
Consequently, we conclude that
\begin{align*}
    |B_{1} \cap T^{-}_{1} \cap \{x \in B_1 : \mathcal{M}(|f|^n)(x) \leq \epsilon\}| &\geq |B_{1} \cap T^{-}_{1}| - |\{x \in B_1 : \mathcal{M}(|f|^n)(x) > \epsilon\}|\\
    &\geq \theta_0|B_1|,
\end{align*}
which proves the lemma.
\end{proof}

\begin{lem} \label{meas}
    Under the assumptions and conclusions of Lemma \ref{key2}, we can find universal constants $\mu_0 \in (0,1)$, $\epsilon \in (0,1)$ and $M>1$ such that
    \begin{align*}
        |B_1 \setminus F_k| \leq \mu_0|B_1 \setminus E_k|
    \end{align*}
    holds for all $k=0,1,2,...$, where
    \begin{align*}
    E_k &= B_1 \cap T_{M^k} \cap \left\{x \in B_1 : \mathcal{M}(|f|^n)(x) \leq \epsilon M^{k(1+\gamma)n} \right\}, \\
    F_k &= B_1 \cap T_{M^{k+1}}.
    \end{align*}
\end{lem}

\begin{proof}
We will apply Lemma \ref{covering}.
Let $\theta_0>0$ be the universal constant given in Lemma \ref{key2}.
We choose $\epsilon=\epsilon(\theta_0)$, $M=M(\theta_0)$, $\theta=\theta(\theta_0)$ and universal $\Theta$ such that $\theta <\min{(\theta_{0}, \Theta)}$ in Lemma \ref{Key}, and then choose $\epsilon_1=\epsilon_1(\epsilon)$ in Lemma \ref{key2}.

Since $E_k \supset B_{1} \cap T^{-}_{1} \cap \{x \in B_1 : \mathcal{M}(|f|^n)(x) \leq \epsilon\}$, we have $|E_k| \geq \theta_0 |B_1| > \theta |B_1|$ by Lemma \ref{key2}.

Fix $k \in \mathbb{N}_0$.
For any ball $B=B_r(x_0) \subset B_1$, we claim that if $|B \cap E_k| \geq \theta|B|$, then $|B \cap F_k| \geq \Theta|B|$.

To prove the claim, we consider the scaled function $\tilde{u} : \Omega \rightarrow \mathbb{R}$ defined by
\begin{align*}
    \tilde{u}(y) = \frac{1}{r^{1+\alpha}M^k}u(ry + x_0),
\end{align*}
where $\Omega$ is the image of $B_2$ under the transformation $x=ry+x_0$.
Note that $B_2 \subset \Omega$ since $2B \subset B_2$, and $\tilde{u}$ satisfies the following inequality
\begin{align*}
    |D\tilde{u}(y)|^{\gamma} \mathcal{P}^{-}(D^{2}\tilde{u}(y)) &= \left|\frac{1}{r^{\alpha}M^k}Du \right|^{\gamma} \mathcal{P}^{-}\left(\frac{r^{1-\alpha}}{M^k}D^{2}u\right) \\
    &=\frac{1}{M^{k(1+\gamma)}}|Du|^{\gamma} \mathcal{P}^{-}(D^{2}u) \\
    &\leq \frac{1}{M^{k(1+\gamma)}} f(ry + x_0) =: \tilde{f}(y) \quad \text{ in } B_2.
\end{align*}
Applying Lemma \ref{Key} to $\tilde{u}$ and $\tilde{f}$, we see that if $|B_1 \cap T_{M}(\tilde{u})| \geq \theta|B_1|$, then $|B_1 \cap T_{1}(\tilde{u}) \cap \{\mathcal{M}(|\tilde{f}|^n) \leq \epsilon\}| \geq \Theta|B_1|$.

To scale back, we first observe that the maximal function is scaling invariant, that is, 
$$\mathcal{M}(|f|^n)(x) = M^{k(1+\gamma)n}\mathcal{M}(|\tilde{f}|^n)(y).$$ For any $C^{1,\alpha}$ cone $P_{M^k}(x)$ with opening $M^k$, consider $\tilde{P}_1(y) = \frac{1}{r^{1+\alpha}M^k}P(ry + x_0)$ which is the scaled function of $P_{M^k}$ with the same scaling of $u$.
Note that the opening of $\tilde{P}_1$ is just $1$ which is independent of $r$.
Thus, we see that
$$P_{M^k}(z) \overset{x}{\leq} u(z) \ \text{ in } B_2 \ \Longleftrightarrow \ \tilde{P}_1(z) \overset{y}{\leq} \tilde{u}(z) \ \text{ in } \Omega. $$
Therefore, for the transformation $x=ry+x_0$ we have
\begin{align*}
    x \in B \cap E_k = B \cap T_{M^k}(u) \cap \{\mathcal{M}(|f|^n) \leq \epsilon M^{k(1+\gamma)n}\} \ &\Longleftrightarrow \ y \in B_{1} \cap T_{1}(\tilde{u}) \cap \{\mathcal{M}(|\tilde{f}|^n) \leq \epsilon\}, \\
    x \in B \cap F_k = B \cap T_{M^{k+1}}(u) \ &\Longleftrightarrow \ y \in B_{1} \cap T_{M}(\tilde{u}).
\end{align*}
Consequently, we deduce that if $|B \cap E_k| \geq \theta|B|$, then $|B \cap F_k| \geq \Theta|B|$,
We then conclude from Lemma \ref{covering} that
\begin{align*}
    |B_1 \setminus F_k| \leq \left(1-\frac{\Theta-\theta}{5^n}\right)|B_1 \setminus E_k|=:\mu_0|B_1 \setminus E_k|,
\end{align*}
which proves the lemma.
\end{proof}

\begin{cor} \label{decay}
    Under the assumptions and conclusions of Lemma \ref{key2}, we obtain
    \begin{align*}
        |B_1 \setminus T_{M^k}| \leq C \mu^k,
    \end{align*}
    where $C=C(n,\lambda,\Lambda,\gamma)>0$,  $0<\mu=\mu(n,\lambda,\Lambda,\gamma)<1$.
\end{cor}

\begin{proof}
Define
\begin{align*}
    \alpha_k &= |B_1 \setminus T_{M^k}|, \\
    \beta_k &= |\{x \in B_1 : \mathcal{M}(|f|^n)(x) \geq \epsilon M^{k(1+\gamma)n} \}|.
\end{align*}
Applying Lemma \ref{meas}, we get $\alpha_{k+1} \leq \mu_0(\alpha_k + \beta_k)$. Hence,
\begin{align*}
    \alpha_{k} \leq \mu_0^k|B_1| + \sum_{i=0}^{k-1} \mu_0^{k-i} \beta_i.
\end{align*}
By the weak type $(1,1)$ property for the maximal operator, we have 
\begin{align*}
    \beta_k \leq C(\epsilon_2  M^{k(1+\gamma)n} )^{-1} \norm{f}^n_{L^n(B_1)} \leq CM^{-k(1+\gamma)n}.
\end{align*}
Therefore, we obtain
\begin{align*}
    \sum_{i=0}^{k-1} \mu_0^{k-i} \beta_i &\leq C \sum_{i=0}^{k-1} \mu_0^{k-i} M^{-k(1+\gamma)ni} \\
    &\leq C \sum_{i=0}^{k-1} \mu_1^{k} \leq Ck\mu_1^k
\end{align*}
with $\mu_1 = \max(\mu_0, M^{-(1+\gamma)n}) \in (0,1).$ Consequently, we have
\begin{align*}
    \alpha_k \leq \mu_0^k|B_1| + Ck\mu_1^k \leq C\mu^k,
\end{align*}
where $\mu := \frac{1+\mu_1}{2} \in (\mu_1,1)$, for some constant $C$ depending only on $n,\lambda,\Lambda$ and $\gamma$.
This completes the proof.
\end{proof}

We now prove Theorem \ref{tangent} using Corollary \ref{decay}.
\begin{proof}
Since $u$ and $-u$ satisfy the assumptions of Lemma \ref{key2}, we get $|B_1 \setminus T^{\pm}_{M^k}| \leq C \mu^k$.
Therefore, $$|B_1 \setminus T_{M^k}| = |B_1 \setminus (T^{-}_{M^k} \cap T^{+}_{M^k})| \leq |B_1 \setminus T^{+}_{M^k}|+ |B_1 \setminus T^{-}_{M^k}| \leq C\mu^k.$$ Consequently, we obtain
\begin{align*}
    |B_1 \setminus T_{t}| \leq Ct^{-\sigma}, \quad \forall t>0,
\end{align*}
where $\sigma = -\log_M \mu > 0$, which completes the proof.
\end{proof}

\providecommand{\bysame}{\leavevmode\hbox to3em{\hrulefill}\thinspace}
\providecommand{\MR}{\relax\ifhmode\unskip\space\fi MR }
% \MRhref is called by the amsart/book/proc definition of \MR.
\providecommand{\MRhref}[2]{%
	\href{http://www.ams.org/mathscinet-getitem?mr=#1}{#2}
}
\providecommand{\href}[2]{#2}

\end{document}